\documentclass[10pt]{article}
\usepackage{amsmath}
\usepackage{amsthm}
\usepackage[a4paper]{geometry}
\usepackage{microtype}
\usepackage[
    pdfusetitle,
    bookmarks=true,
    bookmarksnumbered=false,
    bookmarksopen=false,
    breaklinks=false,
    pdfborder={0 0 1},
    backref=false,
    colorlinks=false,
    hypertexnames=false,
    hidelinks
]{hyperref}
\usepackage[capitalise,noabbrev]{cleveref}
\makeatletter
\usepackage{amssymb}
\usepackage[dvipsnames]{xcolor}
\usepackage{algorithm,algpseudocode}
\newcounter{algsubstate}
\renewcommand{\thealgsubstate}{\alph{algsubstate}}
\newenvironment{algsubstates}
  {\setcounter{algsubstate}{0}%
   \renewcommand{\State}{%
     \stepcounter{algsubstate}%
     \Statex {\footnotesize\thealgsubstate:}\space}}
  {}
\makeatletter
\newcommand{\StateCont}{%
  \Statex
  \ifdim\ALG@tlm>\z@
    \hspace{\dimexpr\ALG@tlm-\algorithmicindent\relax}%
  \fi}
\makeatother
\usepackage{graphicx}

\usepackage{tikz}
\usetikzlibrary{arrows.meta,positioning,calc}
\makeatother
 
\usepackage{subcaption}
\usepackage{makecell}
\usepackage{multirow}
\usepackage{booktabs}
\usepackage{listings}
\usepackage[
    style=numeric,
    backend=biber,
    sortcites,
    hyperref=true,
    maxbibnames=99
]{biblatex}
\usepackage{authblk}

\newtheorem{theorem}{Theorem}[section]
\newtheorem{lemma}[theorem]{Lemma}

\begin{document}
\title{A Memory-Efficient \\ Adjoint State Optimization Method \\ Based on Time-Reversible \\ Dynamical Low-Rank Approximation}
\author[1]{Lukas Einkemmer}
\author[1]{Julian Mangott}
\affil[1]{Department of Mathematics, Universität Innsbruck, Austria}
\date{}
\setcounter{Maxaffil}{0}
\renewcommand\Affilfont{\itshape\small}
\maketitle

\begin{abstract}
    The primary challenge of conducting PDE-constrained optimization for high-dimensional problems, such as kinetic equations, is the often prohibitive memory cost. Computing gradients using the adjoint state method would require the storage of the entire time history of the forward solution. For such problems, where the memory cost for storing a single instance of the forward solution can already be a limiting factor, this is clearly not feasible. In this paper, we propose a memory-efficient adjoint state method that compresses the forward and adjoint solution with a dynamical low-rank approximation (a model order reduction technique) and bypasses the need to store the entire forward solution by employing a time-reversible low-rank integrator. The dynamical low-rank approach introduces a number of challenges: reversibility can fail in the rank-deficient case and the low-rank trajectories can show chaotic behavior. In particular, the latter has a number of important consequences for the optimization problem. We address those challenges and show that our method can drastically reduce the memory requirement for gradient-based optimization of kinetic equations. In particular, we consider two examples from kinetic plasma physics: optimizing beam profiles to suppress a bump-on-tail instability and shaping a particle beam using external electric fields.
\end{abstract}

\section{Introduction}\label{sec:introduction}
PDE-constrained optimization is ubiquitous in applied mathematics, physics and engineering. If $f_\theta$ is the solution of a PDE depending on $d$ parameters $\theta \in \Theta \subset \mathbb{R}^d$, then the goal of PDE-constrained optimization is to minimize the loss functional $J(f_\theta)$ with respect to the parameters $\theta$. Such problems appear, for example, in aerodynamic shape optimization, inverse problems, parameter estimation of stochastic processes, and control and design of fusion devices.

The loss landscapes of PDE-constrained optimization are often very complicated and exhibit many local minima \cite{Einkemmer_2024a,Einkemmer_2024b,Guerra_2025}. Therefore, it is advantageous to use hybrid optimization by combining global and local methods. A global method, such as a genetic algorithm or differential evolution, can identify promising candidate parameters and reduce the risk of convergence to poor local minima. These candidates can then be refined by a local, gradient-based method. In this paper, we focus on the efficient computation of the gradients required for the local optimization stage.

A straightforward way to approximate the gradient is to use centered finite differences,
\begin{equation}\label{eq:finite-differences}
    (\nabla_\theta J(f_\theta))_i \approx \frac{J(f_{\theta + \epsilon e_i}) - J(f_{\theta - \epsilon e_i})}{2\epsilon},
\end{equation}
where $e_i$ is the $i$-th unit vector of the parameter space. In total, the gradient requires $2d$ solutions of the PDE. For high-dimensional parameter spaces, or when solving the PDE is expensive, finite difference gradients quickly become infeasible.

Therefore, the \emph{adjoint state method} is the state-of-the-art method for large-scale PDE-constrained optimization problems. By adding the PDE as a constraint with a Lagrange multiplier, the gradient can be expressed in terms of one forward solution (obtained by solving the PDE) and one adjoint solution (obtained by solving a modified PDE -- the so-called \emph{adjoint equation} -- backward in time). Thus, the cost of one gradient evaluation with the adjoint state method is independent of the number of parameters.

However, this comes at the cost of high memory requirements. The adjoint equation is solved backward in time and often depends on the forward solution at all time points. A conventional implementation therefore stores the entire forward trajectory. This can be prohibitively expensive, especially for PDEs whose solution is posed in a high-dimensional phase space. The Vlasov--Poisson equation, considered for the numerical examples in this paper, is a prototypical hyperbolic example of such a PDE.

We address the high memory requirements of adjoint methods with a novel memory-efficient scheme based on two ingredients. The first is the \emph{dynamical low-rank approximation} (DLRA), which represents the phase-space distribution by time-dependent low-rank factors (see, e.g., \cite{Koch_2008,Lubich_2014}). Instead of solving the PDE for the full high-dimensional solution, lower-dimensional factors and a small coefficient matrix are evolved. Thus, both memory consumption and computational cost are significantly reduced. DLRA has been used in \cite{Scalone_2025}, where the authors use a finite difference approximation of the gradient to optimize an elliptic problem. The first use of DLRA for the adjoint state method is in \cite{Baumann_2025} to solve inverse problems for radiative transfer. This approach is close in spirit to our work, but it still requires the storage of the entire forward solution, which is very costly in terms of memory. 

This motivates the second ingredient of our memory-efficient scheme, a time-reversible DLRA. Such a time-reversible scheme guarantees that we accurately restore the forward solution trajectory in reverse order if we integrate the forward solution backward in time. Consequently, the forward states needed by the adjoint are recomputed during the backward pass instead of being stored. In principle, this can reduce the memory requirement from the full forward trajectory to a single instance of the forward and adjoint solution at any point in time. Our time-reversible DLRA is based on the projector-splitting integrator. We note that all the subflows of this integrator should be computed in a time-reversible manner. We show how this can be done efficiently for the Vlasov--Poisson (and other kinetic) equations by storing lower-dimensional quantities (such as the electric field) in the forward solve.

In addition, we provide some interesting observations of robust DLRA integrators. In particular, we show that the theoretical symmetry of the Strang projector-splitting integrator can fail in the rank-deficient case and that DLRA integrators can have chaotic behavior even if this is not true for the underlying PDE. Both present challenges for performing memory-efficient optimization in a DLRA framework that we will address in this paper.

The remainder of the paper is structured as follows. In \cref{sec:theoretical-background}, we recapitulate the necessary theoretical background concerning the adjoint state method and the DLRA. In \cref{sec:memory-efficient-adjoint-method}, we present the memory-efficient adjoint strategy and discuss limitations caused by rank deficiency and chaotic behavior of the DLRA. Moreover, we address how the DLRA affects the landscape of the loss functional. In \cref{sec:algorithm-numerical-experiments}, we describe our algorithm, which uses a checkpointing strategy to overcome these limitations. The effectiveness of our method is demonstrated by two optimization problems from kinetic plasma physics. Finally, we conclude in \cref{sec:conclusion}.

\section{Theoretical background}\label{sec:theoretical-background}
In this section, we review the theoretical background needed for developing our memory-efficient scheme for the adjoint state method. To this end, we describe first the adjoint state method for PDE-constrained optimization problems. We then briefly summarize the key concepts of the dynamical low-rank approximation and the projector-splitting integrator.

\subsection{Adjoint state method}\label{sec:adjoint-method}
In this section, we briefly review the adjoint state method for a parameter-dependent PDE of the form
\begin{equation}\label{eq:generic_pde}
    \partial_t f_{\theta}(t, z) = F(t, f_{\theta}(t, \cdot), \theta)(z), \qquad f_{\theta}(t=0, z) = f_0(z)
\end{equation}
on a finite time interval $t \in [0, T]$ with the phase space variable $z \in \Omega$. For each $t$, the right-hand side is given by a possibly nonlinear operator $F$ that depends explicitly on time $t$, on the solution $f$ and on the parameters $\theta \in \Theta \subset \mathbb{R}^d$, where $d$ is the number of optimization parameters. Moreover, we assume that the operator $F$ lives on a suitable domain that fulfills the required regularity and boundary conditions. In particular, $F$ contains no time derivatives of $f$, but acts on $f(t, z)$ through its phase-space variables $z$.

Our aim is to find the parameters $\theta$ that minimize a loss functional $J$, i.e.
\begin{equation}\label{eq:optimization}
    \min_{\theta} J(f_{\theta})
\end{equation}
with the constraint that $f_{\theta}$ is a solution of \cref{eq:generic_pde} with parameters $\theta$. We assume that the loss functional $J$ does not explicitly depend on the parameters $\theta$, but consists of a running and a terminal contribution, i.e.
\begin{equation}
    J(f_{\theta}) = \int_0^T \mathcal{J}(f_{\theta}) \, \mathrm{d}t + \Phi(f_{\theta}(T)).\label{eq:abstract-loss}
\end{equation}

As outlined in the introduction, we want to make use of local optimization with a gradient-based method. Due to the implicit dependence of $J$ on the parameters $\theta$ through the PDE constraint, the computation of the gradient is however often not feasible. We have seen in the introduction that the computation with a finite difference scheme scales with the number of parameters $d$ and is therefore very expensive to compute if the number of parameters is large. The same is true for automatic differentiation, which has to be carried out for every single time step. Even if automatic differentiation is supported by the numerical solver for the PDE, it can be prohibitively expensive to compute due to large computational graphs.

Therefore, we solve the constrained optimization problem with the adjoint state method. We introduce Lagrangian multipliers $g$ for the PDE and $u$ for the initial condition and define the Lagrangian
\begin{equation}
    L(f, g, u, \theta) = J(f) - \int_0^T \langle g, \partial_t  f - F(t, f, \theta) \rangle \, \mathrm{d}t - \langle u, f(t=0) - f_0 \rangle,
    \label{eq:abstract-lagrangian}
\end{equation}
where $\langle \cdot, \cdot\rangle$ denotes the inner product on the phase space $\Omega$. The original constrained optimization problem~(\ref{eq:optimization}) is recast to an unconstrained formulation with $f$ separated from the implicit dependence on $\theta$, i.e.
\begin{equation}
    \min_{f, g, u, \theta} L(f, g, u, \theta).
\end{equation}
The first-order optimality condition requires
\begin{equation}\label{eq:optimality-condition}
    \nabla_{\theta} L = 0, \qquad D_f(L)(h) = 0, \qquad D_g(L)(k) = 0, \qquad \text{and} \qquad D_u(L)(v) = 0
\end{equation}
for all variations $h$, $k$ and $v$. Here, $D_f(L)(h)$ denotes the Fréchet derivative of $L$ at $f$ acting on the arbitrary variation $h$.

The second condition corresponds to the so-called adjoint equation which the Lagrange multiplier $g$ has to satisfy. We derive the adjoint equation by computing the Fréchet derivative and moving the time derivative with integration by parts from $f$ to $g$:
\begin{align*}
    D_f(L)(h) &= \int_0^T D_f(\mathcal{J})(h) \, \mathrm{d}t + D_f(\Phi)(h(T)) - \langle g(T), h(T) \rangle + \langle g(0), h(0) \rangle \\ & \qquad + \int_0^T \langle \partial_t g + (D_f(F))^*(g), h \rangle \, \mathrm{d}t - \langle u, h(0) \rangle,
\end{align*}
where $(D_f(F))^*$ denotes the adjoint of the Fréchet derivative. Note that here we assumed that the boundary conditions in phase space are chosen such that the corresponding boundary terms resulting from the integration by parts vanish.

Since $D_f(L)(h) = 0$ for any variation $h$ due to the first-order optimality condition, we choose a variation with vanishing endpoints, $h(0) = h(T) = 0$, and obtain the adjoint equation
\begin{equation}
    \partial_t g + (D_f(F))^*(g) = -\nabla_f \mathcal{J},\label{eq:generic_adjoint}
\end{equation}
where $\nabla_f \mathcal{J}$ is the Riesz representative of the Fréchet derivative defined by $\langle \nabla_f \mathcal{J}, h \rangle = D_f(\mathcal{J})(h)$. By choosing a variation $h$ with nonvanishing boundaries at $T$, we obtain the terminal condition for the adjoint equation,
\begin{equation}
    g(T) = \nabla_{f(T)} \Phi.\label{eq:generic_terminal_condition}
\end{equation}
Thus, the adjoint equation is a terminal-value problem that is solved backward in time.

If we know the Lagrange multiplier $g$ and if $f$ fulfills \cref{eq:generic_pde} for given parameters $\theta$, then the first condition~\eqref{eq:optimality-condition} enables us to compute the gradient by
\begin{equation}
    \nabla_{\theta} J = \nabla_{\theta} L = \int_0^T \langle g(t,\cdot), \nabla_{\theta} F(t, f, \theta)(\cdot) \rangle \, \mathrm{d}t.\label{eq:generic_gradient}
\end{equation}

The adjoint state method proceeds as follows. For given parameters $\theta$, we first compute the solution of \cref{eq:generic_pde} forward in time. In the following, we will call \cref{eq:generic_pde} the \emph{forward problem} or \emph{forward equation} and its solution the \emph{forward solution}. We then use the forward solution at the final time point $T$ to compute the terminal condition of $g$ via \cref{eq:generic_terminal_condition}. Starting from the terminal condition, we then integrate the adjoint equation~\eqref{eq:generic_adjoint} backward in time. Finally, we compute the gradient $\nabla_{\theta} J$ from $f$ and $g$ with \cref{eq:generic_gradient}. This is illustrated in \cref{fig:adjoint-method}, where $\varphi^f_{\Delta t}$ and $\varphi^g_{\Delta t}$ denote the flow of the numerical scheme with time step size $\Delta t$ for the forward and the adjoint problem, respectively.

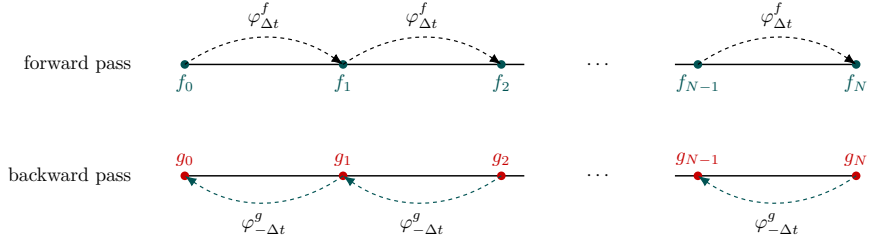
\begin{figure}[!htb]
    \centering
    \resizebox{0.72\textwidth}{!}{\begin{tikzpicture}[
    scale=0.92,
    every node/.style={inner sep=1pt},
    forwarddot/.style={circle, fill=teal!65!black, inner sep=0pt, minimum size=4.4pt},
    adjointdot/.style={circle, fill=red!78!black, inner sep=0pt, minimum size=4.4pt},
    timeline/.style={black, line width=0.75pt},
    steparrow/.style={-{Latex[length=2.1mm,width=1.7mm]}, black, line width=0.6pt, dashed, dash pattern=on 2pt off 2pt},
    reversearrow/.style={-{Latex[length=2.1mm,width=1.7mm]}, teal!65!black, line width=0.6pt, dashed, dash pattern=on 2pt off 2pt}
]
    \def\top{2.18}
    \def\bottom{0}

    \coordinate (f0) at (0,\top);
    \coordinate (f1) at (3.10,\top);
    \coordinate (f2) at (6.20,\top);
    \coordinate (fa) at (10.05,\top);
    \coordinate (fb) at (13.15,\top);

    \coordinate (g0) at (0,\bottom);
    \coordinate (g1) at (3.10,\bottom);
    \coordinate (g2) at (6.20,\bottom);
    \coordinate (ga) at (10.05,\bottom);
    \coordinate (gb) at (13.15,\bottom);

    \draw[timeline] (f0) -- (f1) -- (f2) -- ++(0.45,0);
    \draw[timeline] ($(fa)+(-0.45,0)$) -- (fa) -- (fb);
    \draw[timeline] (g0) -- (g1) -- (g2) -- ++(0.45,0);
    \draw[timeline] ($(ga)+(-0.45,0)$) -- (ga) -- (gb);

    \node[anchor=east, font=\normalsize] at (-1.0,\top) {forward pass};
    \node[anchor=east, font=\normalsize] at (-1.0,\bottom) {backward pass};

    \node[font=\normalsize] at (8.125,\top) {$\cdots$};
    \node[font=\normalsize] at (8.125,\bottom) {$\cdots$};

    \foreach \p in {f0,f1,f2,fa,fb} {\node[forwarddot] at (\p) {};}
    \foreach \p in {g0,g1,g2,ga,gb} {\node[adjointdot] at (\p) {};}

    \draw[steparrow] (f0) to[out=35,in=145] node[midway, above=4pt, font=\normalsize] {$\varphi^f_{\Delta t}$} (f1);
    \draw[steparrow] (f1) to[out=35,in=145] node[midway, above=4pt, font=\normalsize] {$\varphi^f_{\Delta t}$} (f2);
    \draw[steparrow] (fa) to[out=35,in=145] node[midway, above=4pt, font=\normalsize] {$\varphi^f_{\Delta t}$} (fb);

    \draw[reversearrow] (g1) to[out=-145,in=-35] node[midway, below=4pt, font=\normalsize, black] {$\varphi^g_{-\Delta t}$} (g0);
    \draw[reversearrow] (g2) to[out=-145,in=-35] node[midway, below=4pt, font=\normalsize, black] {$\varphi^g_{-\Delta t}$} (g1);
    \draw[reversearrow] (gb) to[out=-145,in=-35] node[midway, below=4pt, font=\normalsize, black] {$\varphi^g_{-\Delta t}$} (ga);

    \node[teal!65!black, font=\normalsize, below=4pt] at (f0) {$f_{0}$};
    \node[teal!65!black, font=\normalsize, below=4pt] at (f1) {$f_{1}$};
    \node[teal!65!black, font=\normalsize, below=4pt] at (f2) {$f_{2}$};
    \node[teal!65!black, font=\normalsize, below=4pt] at (fa) {$f_{N-1}$};
    \node[teal!65!black, font=\normalsize, below=4pt] at (fb) {$f_{N}$};

    \node[red!78!black, font=\normalsize, above=4pt] at (g0) {$g_{0}$};
    \node[red!78!black, font=\normalsize, above=4pt] at (g1) {$g_{1}$};
    \node[red!78!black, font=\normalsize, above=4pt] at (g2) {$g_{2}$};
    \node[red!78!black, font=\normalsize, above=4pt] at (ga) {$g_{N-1}$};
    \node[red!78!black, font=\normalsize, above=4pt] at (gb) {$g_{N}$};
\end{tikzpicture}}
    \caption{Schematic of the conventional adjoint state method. The forward and the adjoint solution are denoted by $f$ and $g$, respectively. The forward problem is solved with the numerical flow $\varphi^f$ and the adjoint problem with $\varphi^g$.}
    \label{fig:adjoint-method}
\end{figure}

We close this section on the adjoint state method by mentioning the two standard approaches to derive and implement adjoint methods. In this work, we follow the \emph{optimize-then-discretize} approach, where first the continuous adjoint equation is derived and, for the numerical simulation, subsequently both the forward and adjoint equations are discretized. This approach keeps the discretization of the forward and the adjoint equation completely flexible. However, it is not guaranteed that the resulting discrete gradient is exactly the gradient of the discrete loss functional.

Alternatively, the adjoint equation is derived from the discretized forward problem, called \emph{discretize-then-optimize}. This yields a gradient that is consistent with the implemented numerical scheme. The drawback is that the discrete adjoint can be cumbersome to derive and may strongly depend on implementation details.

The computational cost for the gradient computed by the adjoint method is independent of the number of parameters and only requires two PDE solves. However, a central difficulty remains: computing the gradient of the loss functional $\nabla_{\theta} J$ with \cref{eq:generic_gradient} or the adjoint $g$ with the adjoint equation~\eqref{eq:generic_adjoint}, in general, requires us to store the full forward solution. This is the case when $\nabla_f \mathcal{J}$ depends on $f$ (see the beam heating problem for kinetic plasma physics in \cref{sec:beam-heating}) or $F$ contains a parameter-dependent advection term $H_\theta \partial_z f$ (see the beam shaping problem in \cref{sec:beam-shaping}). If the phase space is high-dimensional, storing $f(t, z)$ for a single time point can already be very expensive due to the curse of dimensionality. Storing $f(t,z)$ for all time steps is clearly infeasible for high-dimensional problems.

\subsection{Dynamical low-rank approximation}\label{sec:DLRA}
The main idea of the dynamical low-rank approximation (DLRA) is to approximate a time-dependent high-dimensional function by a linear combination of time-dependent lower-dimensional functions called \emph{low-rank factors}. We will describe how to use this technique to compress the solution of \cref{eq:generic_pde}, $f(t, z)$, which depends on time $t$ and on the phase space variable $z \in \Omega$. We will use the same approach for the backward solution, i.e., for obtaining $g$.

In order to approximate $f(t, z)$ with low-rank factors, we write the phase space as the direct sum of two lower-dimensional subspaces $\Omega_x$ and $\Omega_v$, i.e.~$\Omega = \Omega_x \oplus \Omega_v$. Thus, we can express a point in the full phase space, $z \in \Omega$, by $z = (x, v)$, where $x \in \Omega_x$ and $v \in \Omega_v$. For kinetic equations this assumption is well justified, as $\Omega_x$ can still have an arbitrary geometry and $\Omega = \mathbb{R}^{d_v}$, where $d_v \leq 3$ is the number of dimensions in velocity space, or some suitable truncation. We then approximate $f$ as follows 
\begin{equation}\label{eq:dlra}
    f(t, x, v) \approx \sum_{i,j=1}^{r} X_i(t, x) S_{ij}(t) V_j(t, v),
\end{equation}
where $r$ is called the \emph{rank}, $S(t)\in\mathbb{R}^{r\times r}$ is the time-dependent \emph{coefficient matrix}, and $X_i(t, \textcolor{blue}{\cdot}) \in L^2(\Omega_x)$ and $V_j(t, \textcolor{blue}{\cdot}) \in L^2(\Omega_v)$ are time-dependent, lower-dimensional basis functions that are orthonormal
\begin{equation*}
    \langle X_i, X_j \rangle_x = \delta_{ij}, \qquad \langle V_i, V_j \rangle_v = \delta_{ij},
\end{equation*}
and fulfill the following gauge conditions
\begin{equation*}
    \langle X_i, \partial_t X_j \rangle_x=0, \qquad \langle V_i, \partial_t V_j \rangle_v = 0.
\end{equation*}
Here, $\langle \cdot, \cdot \rangle_x$ and $\langle \cdot, \cdot \rangle_v$ denote the standard $L^2$ inner products on $\Omega_x$ and $\Omega_v$, respectively.

For numerical simulations, we typically discretize our problem by $N_x$ and $N_v$ grid points in the $x$- and $v$-direction, respectively. For a fixed time $t$, this results in discretized basis functions $X \in \mathbb{R}^{N_x \times r}$ and $V \in \mathbb{R}^{N_v \times r}$. In this setting, \cref{eq:dlra} has the form of a singular value decomposition (SVD), with the difference that $S$ is allowed to be a non-diagonal matrix. For a kinetic model with $3$ spatial and $3$ velocity dimensions that is discretized with $N_x = N_v = N$ grid points, the memory consumption is reduced from $\mathcal{O}(N^6)$ to $\mathcal{O}(r N^3)$.

In the following, we derive evolution equations for the low-rank factors $X$, $S$ and $V$. We achieve this by projecting the right-hand side of \eqref{eq:generic_pde} on the tangent space of the manifold of low-rank functions of the form~\eqref{eq:dlra}. The projector $P$ at position $f = \sum_{ij=1}^r X_i S_{ij} V_j$ has the form
\begin{equation}\label{eq:projector}
    P(f) g = P_{\overline{V}} g - P_{\overline{X}} P_{\overline{V}} g + P_{\overline{X}} g,
\end{equation}
where $P_{\overline{X}}$ and $P_{\overline{V}}$ are the projectors onto $\overline{X} = \mathrm{span}\{ X_i \}$ and $\overline{V} = \mathrm{span}\{ V_i \}$, respectively.

The dynamical low-rank approximation does not solve \cref{eq:generic_pde}, but the projected dynamics given by
\begin{equation}\label{eq:projected-PDE}
    \partial_t f(t, z) = P(f) F(t, f(t, \cdot), \theta)(z),
\end{equation}
where $z = (x, v)$. The projector $P$ ensures that the rank will not increase when the solution evolves in time. The origins of this approach stem from quantum mechanics \cite{Dirac_1930,Frenkel_1934}. Later, a numerical integrator for matrix ordinary differential equations (ODEs) was constructed from \cref{eq:projected-PDE} by \cite{Koch_2008}. This integrator for the dynamical low-rank approximation required to compute the inverse of $S$, which is not robust due to the small singular values in $S$. However, small singular values are required for a good low-rank approximation. Only relatively recently, robust integrators have been developed that do not suffer from small singular values. In the following, we will describe one of these robust integrators, the \emph{projector-splitting integrator} proposed for matrix equations in \cite{Lubich_2014} and extended to a continuous PDE framework for kinetic equations in \cite{Einkemmer_2018}. The projector-splitting integrator treats the three different terms in \cref{eq:projected-PDE} separately. That is, we consider the equations
\begin{subequations}\label{eq:psi-equations}
    \begin{align}
        \partial_t f(t, z) &= P_{\overline{V}} F(t, f(t, \cdot), \theta), \tag{\theequation K} \label{eq:psi-equations-K} \\
        \partial_t f(t, z) &= -P_{\overline{X}} P_{\overline{V}} F(t, f(t, \cdot), \theta), \tag{\theequation S} \label{eq:psi-equations-S} \\
        \partial_t f(t, z) &= P_{\overline{X}} F(t, f(t, \cdot), \theta). \tag{\theequation L} \label{eq:psi-equations-L}
    \end{align}
\end{subequations}
The corresponding numerical flows are denoted by $\psi^{K}_t$, $\psi^{S}_t$ and $\psi^{L}_t$ for equations \cref{eq:psi-equations-K,,eq:psi-equations-S,,eq:psi-equations-L}, respectively. Then we can compute the solution with any splitting method such as Lie--Trotter splitting
\begin{equation}\label{eq:lie-trotter}
    \begin{aligned}
        \varphi^1_{\Delta t} &= \psi^{L}_{\Delta t} \circ \psi^{S}_{\Delta t} \circ \psi^{K}_{\Delta t},
    \end{aligned}
\end{equation}
or Strang splitting
\begin{equation}\label{eq:strang}
    \varphi^2_{\Delta t} = \psi^{K}_{\Delta t / 2} \circ \psi^{S}_{\Delta t / 2} \circ \psi^{L}_{\Delta t} \circ \psi^{S}_{\Delta t / 2} \circ \psi^{K}_{\Delta t / 2}.
\end{equation}

The practical algorithm for the Lie--Trotter splitting $\varphi^1_{\Delta t}$ is shown in \cref{alg:projector-splitting-lie} for the general parameter-dependent PDE~\eqref{eq:generic_pde}. The algorithm for Strang splitting, $\varphi^2_{\Delta t}$, can be obtained similarly. Since there is currently no proof that shows second-order convergence for Strang splitting, we consider it to be only formally second-order in time. In numerical experiments, we observe that the Strang projector-splitting integrator yields more accurate results than the first-order Lie--Trotter scheme.

Furthermore, let us remark that the dynamical low-rank approximation exclusively works with the low-rank factors and never forms the full solution, and thus circumvents the curse of dimensionality. To accomplish this, the projections in \cref{eq:practical-psi-equations} have to be computed efficiently using the specific form of $F$ for a given problem. It has been shown that this can be done for a variety of kinetic equations; for more information on the Vlasov--Poisson and Vlasov--Maxwell equations see, e.g., \cite{Einkemmer_2018,Piazzola_2020}, for the radiative transfer equation, see, e.g., \cite{Peng_2020,Einkemmer_2021}, and for the BGK/Fokker-Planck/Boltzmann equations see, e.g., \cite{Einkemmer_2021d,Coughlin_2022,Hu_2022,Zhang_2025}.

What makes the projector-splitting integrator robust is how the basis functions are updated. They are indirectly obtained from the auxiliary quantities $K$ and $L$ via orthogonalization. This is also the reason why the three steps of the algorithm are commonly called K, S and L steps. In the discretized setting, this orthonormalization can be achieved with a QR or polar decomposition. Forming the auxiliary quantities $K$ and $L$ avoids the computation of the inverse of $S$ and makes the integrator well conditioned with respect to small singular values \cite{Kieri_2016}.

\begin{algorithm}
    \caption{Projector-splitting integrator $\varphi^1_{\Delta t}$ (Lie--Trotter splitting)}\label{alg:projector-splitting-lie}
    \begin{algorithmic}[1]
        \Require Low-rank factors $X_i^n(x)$, $S_{ij}^n$, $V_j^n(v)$ at time $t_n$, time step size $\Delta t$.
        \Ensure Low-rank factors $X_i^{n+1}(x)$, $S_{ij}^{n+1}$, $V_j^{n+1}(v)$ at time $t_{n+1} = t_n + \Delta t$.
        \State \textbf{K step:} Form $K_i(t_n, x) = \sum_{j=1}^{r} X_j^n(x) S_{ji}^n$.
        \begin{algsubstates}
            \State Integrate from $t = t_n$ to $t_{n+1}$ the PDE
            \refstepcounter{equation}\label{eq:practical-psi-equations}%
            \begin{equation} \tag{\theequation K} \label{eq:practical-psi-equations-K}
                \partial_t K_i(t, x) = \left\langle V_i^n, F\left( t, \sum_{l=1}^{r} K_l(t, x) V_l^n, \theta \right) \right\rangle_v.
            \end{equation}
            \State Orthonormalize $K_j(t_{n+1}, x)$ to obtain $X^{n+1}_i(x)$ and $S^{*}_{ij}$.
        \end{algsubstates}
        \State \textbf{S step:} Set $S_{ij}(t_n) = S^*_{ij}$.
        \begin{algsubstates}
            \State Integrate from $t = t_n$ to $t_{n+1}$ the matrix ODE
            \begin{equation} \tag{\theequation S} \label{eq:practical-psi-equations-S}
                \frac{\mathrm{d}}{\mathrm{d}t} S_{ij}(t) = -\left\langle X_i^{n+1} V_j^n, F\left( t, \sum_{k,l=1}^{r} X_k^{n+1} S_{kl}(t) V_l^n, \theta \right) \right\rangle_{x,v}.
            \end{equation}
            \State Set $S^{**}_{ij} = S_{ij}(t_{n+1})$.
        \end{algsubstates}
        \State \textbf{L step:} Form $L_i(t_n, v) = \sum_{j=1}^{r} S_{ij}^{**} V_j^n(v)$.
        \begin{algsubstates}
            \State Integrate from $t = t_n$ to $t_{n+1}$ the PDE
            \begin{equation} \tag{\theequation L} \label{eq:practical-psi-equations-L}
                \partial_t L_i(t, v) = \left\langle X_i^{n+1}, F\left( t, \sum_{l=1}^{r} X_l^{n+1} L_l(t, v), \theta \right) \right\rangle_x.
            \end{equation}
            \State Orthonormalize $L_i(t_{n+1}, v)$ to obtain $V^{n+1}_j(v)$ and $S^{n+1}_{ij}$.
        \end{algsubstates}
    \end{algorithmic}
\end{algorithm}

The DLRA can be combined in several different ways with the discretize-then-optimize or with the optimize-then-discretize approach. In \cite{Baumann_2025}, the optimization was carried out first, then the discretization and finally the low-rank approximation was applied. In this paper, we first optimize, then we apply the low-rank approximation and finally, we discretize the low-rank evolution equations.

\section{Memory-efficient adjoint method}\label{sec:memory-efficient-adjoint-method}

The DLRA described in the previous section enables us to dynamically compress the solution of a high-dimensional PDE. We can combine this with the adjoint state method (see \cref{fig:adjoint-method}) by solving both the forward and backward problem using the DLRA. This already reduces the memory requirements significantly. However, it still requires us to store the entire forward solution in order to compute the adjoint and then the gradient using \cref{eq:generic_gradient}. Especially for relatively long-time simulations (which are very relevant in practice), this can still be prohibitive.

We can avoid the storage of the forward problem as follows: after obtaining $f(T, z)$, we run the forward problem backward in time using the adjoint\footnote{Unfortunately, the nomenclature is confusing here. The adjoint of a numerical method, see \cite{Hairer_2006}, is the method that needs to be applied to obtain the same result when integrating backward in time. This should \emph{not} be confused with the adjoint equation, i.e.~the equation for $g$ in the adjoint state method.} method alongside the backward solver for $g$. At each time step, we then directly update the gradient according to \cref{eq:generic_gradient}. Given the flow of a numerical method $\varphi_{\Delta t}$, the \emph{adjoint} of that flow $\varphi^*_{-\Delta t}$ is defined such that $\varphi^*_{-\Delta t} \circ \varphi_{\Delta t}(f) = f$, see, e.g., \cite{Hairer_2006}. Thus, the two flows cancel if one is applied with time step size $\Delta t$ forward and the other with time step size $-\Delta t$, i.e.~backward in time. This guarantees that we obtain the same result for $f$ and the gradient, without explicitly storing $f$.

For splitting schemes, the adjoint is obtained by reversing the order and using the adjoint of the individual subflows (see, e.g., \cite{Hairer_2006}). For the Lie--Trotter splitting in \cref{eq:lie-trotter}, we have
\[ \varphi^{1, \ast}_{\Delta t} = \psi^{K, \ast}_{\Delta t} \circ \psi^{S, \ast}_{\Delta t} \circ \psi^{L, \ast}_{\Delta t} \]
and for the Strang splitting in \cref{eq:strang}, we have
\[ \varphi^{2, \ast}_{\Delta t} = \psi^{K, \ast}_{\Delta t / 2} \circ \psi^{S, \ast}_{\Delta t / 2} \circ \psi^{L, \ast}_{\Delta t} \circ \psi^{S, \ast}_{\Delta t / 2} \circ \psi^{K, \ast}_{\Delta t / 2}. \]
If we use a symmetric method for the individual subflows (such as the implicit midpoint method for the time integration of a PDE that was appropriately discretized in phase space), then $\psi^{K, \ast}_{\Delta t} = \psi^{K}_{\Delta t}$, etc. In this case the Strang splitting scheme is symmetric, i.e.~$\varphi^{2, \ast}_{\Delta t} = \varphi^2_{\Delta t}$. We will show in \cref{sec:low-rank-steps-reversibility} a practical method for the Vlasov--Poisson system that avoids solving the K, S and L evolution equations together with the Poisson equation with an implicit scheme by storing the appropriate electric field, which is a lower-dimensional quantity.

In the context of DLRA, we can thus use the Lie--Trotter projector-splitting $\varphi^1_{\Delta t}$ for the forward pass and its adjoint, $\varphi^{1,*}_{-\Delta t}$, for the backward pass, or vice versa. Alternatively, Strang splitting as a completely symmetric splitting scheme could be used both for the forward and the backward pass. This approach drastically reduces the memory required and is illustrated in \cref{fig:memory-efficient-adjoint-method}.


\begin{figure}[!htb]
    \centering
    \resizebox{0.72\textwidth}{!}{\begin{tikzpicture}[
    scale=0.92,
    every node/.style={inner sep=1pt},
    forwarddot/.style={circle, fill=teal!65!black, inner sep=0pt, minimum size=4.4pt},
    adjointdot/.style={circle, fill=red!78!black, inner sep=0pt, minimum size=4.4pt},
    timeline/.style={black, line width=0.75pt},
    steparrow/.style={-{Latex[length=2.1mm,width=1.7mm]}, black, line width=0.6pt, dashed, dash pattern=on 2pt off 2pt},
    reversearrow/.style={-{Latex[length=2.1mm,width=1.7mm]}, black, line width=0.6pt, dashed, dash pattern=on 2pt off 2pt}
]
    \def\top{2.55}
    \def\bottom{0}

    \coordinate (f0) at (0,\top);
    \coordinate (f1) at (3.10,\top);
    \coordinate (f2) at (6.20,\top);
    \coordinate (fa) at (10.05,\top);
    \coordinate (fb) at (13.15,\top);

    \coordinate (g0) at (0,\bottom);
    \coordinate (g1) at (3.10,\bottom);
    \coordinate (g2) at (6.20,\bottom);
    \coordinate (ga) at (10.05,\bottom);
    \coordinate (gb) at (13.15,\bottom);

    \draw[timeline] (f0) -- (f1) -- (f2) -- ++(0.45,0);
    \draw[timeline] ($(fa)+(-0.45,0)$) -- (fa) -- (fb);
    \draw[timeline] (g0) -- (g1) -- (g2) -- ++(0.45,0);
    \draw[timeline] ($(ga)+(-0.45,0)$) -- (ga) -- (gb);

    \node[anchor=east, font=\normalsize] at (-1.0,\top) {forward pass};
    \node[anchor=east, font=\normalsize] at (-1.0,\bottom) {backward pass};

    \node[font=\normalsize] at (8.125,\top) {$\cdots$};
    \node[font=\normalsize] at (8.125,\bottom) {$\cdots$};

    \foreach \p in {f0,f1,f2,fa,fb} {\node[forwarddot] at (\p) {};}
    \foreach \p in {g0,g1,g2,ga,gb} {\node[adjointdot] at (\p) {};}

    \draw[steparrow] (f0) to[out=35,in=145] node[midway, above=4pt, font=\normalsize] {$\varphi^f_{\Delta t}$} (f1);
    \draw[steparrow] (f1) to[out=35,in=145] node[midway, above=4pt, font=\normalsize] {$\varphi^f_{\Delta t}$} (f2);
    \draw[steparrow] (fa) to[out=35,in=145] node[midway, above=4pt, font=\normalsize] {$\varphi^f_{\Delta t}$} (fb);

    \draw[reversearrow] (g1) to[out=-145,in=-35] node[midway, below=4pt, font=\normalsize] {$\varphi^{f,*}_{-\Delta t}$, $\varphi^{g}_{-\Delta t}$} (g0);
    \draw[reversearrow] (g2) to[out=-145,in=-35] node[midway, below=4pt, font=\normalsize] {$\varphi^{f,*}_{-\Delta t}$, $\varphi^{g}_{-\Delta t}$} (g1);
    \draw[reversearrow] (gb) to[out=-145,in=-35] node[midway, below=4pt, font=\normalsize] {$\varphi^{f,*}_{-\Delta t}$, $\varphi^{g}_{-\Delta t}$} (ga);

    \node[teal!65!black, font=\normalsize, below=4pt, xshift=-5pt] (lf0) at (f0) {$f_{0}$};
    \node[teal!65!black, font=\normalsize, below=4pt, xshift=-5pt] (lf1) at (f1) {$f_{1}$};
    \node[teal!65!black, font=\normalsize, below=4pt, xshift=-5pt] (lf2) at (f2) {$f_{2}$};
    \node[teal!65!black, font=\normalsize, below=4pt, xshift=-5pt] (lfa) at (fa) {$f_{N-1}$};
    \node[teal!65!black, font=\normalsize, below=4pt, xshift=-5pt] (lfb) at (fb) {$f_{N}$};

    \foreach \p in {lf0,lf1,lf2,lfa,lfb} {
        \draw[black, line width=0.55pt]
            ($(\p.south west)+(-1pt,0)$) -- ($(\p.north east)+(1pt,0)$);
    }

    \node[orange!85!black, font=\normalsize, anchor=west] at ($(lf0.east)+(4pt,0)$) {$E_{0}$};
    \node[orange!85!black, font=\normalsize, anchor=west] at ($(lf1.east)+(4pt,0)$) {$E_{1}$};
    \node[orange!85!black, font=\normalsize, anchor=west] at ($(lf2.east)+(4pt,0)$) {$E_{2}$};
    \node[orange!85!black, font=\normalsize, anchor=west] at ($(lfa.east)+(4pt,0)$) {$E_{n-1}$};
    \node[orange!85!black, font=\normalsize, anchor=west] at ($(lfb.east)+(4pt,0)$) {$E_{n}$};

    \node[font=\normalsize, above=4pt] at (g0) {$f_{0},\, {\color{red!78!black}g_{0}},\, {\color{orange!85!black}E_{0}}$};
    \node[font=\normalsize, above=4pt] at (g1) {$f_{1},\, {\color{red!78!black}g_{1}},\, {\color{orange!85!black}E_{1}}$};
    \node[font=\normalsize, above=4pt] at (g2) {$f_{2},\, {\color{red!78!black}g_{2}},\, {\color{orange!85!black}E_{2}}$};
    \node[font=\normalsize, above=4pt] at (ga) {$f_{N-1},\, {\color{red!78!black}g_{N-1}},\, {\color{orange!85!black}E_{n-1}}$};
    \node[font=\normalsize, above=4pt] at (gb) {$f_{N},\, {\color{red!78!black}g_{N}},\, {\color{orange!85!black}E_{n}}$};
\end{tikzpicture}}
    \caption{Schematic of the memory-efficient adjoint state method. The forward and the adjoint solution are denoted by $f$ and $g$, respectively. The forward pass is solved by $\varphi^f$, the backward pass by the adjoint method $\varphi^{f,*}$ for the forward problem and $\varphi^{g}$ for the adjoint problem. The forward trajectory is thus restored during the backward pass. In the context of DLRA, we use the Lie--Trotter $\varphi^1_{\Delta t}$ or Strang projector-splitting $\varphi^2_{\Delta t}$ for $\varphi^f$ and $\varphi^g$. Only quantities that are needed to reconstruct nonlinear terms of the right-hand side of the PDE, which are typically lower-dimensional (such as the electric field for the Vlasov--Poisson equation), have to be stored.}
    \label{fig:memory-efficient-adjoint-method}
\end{figure}
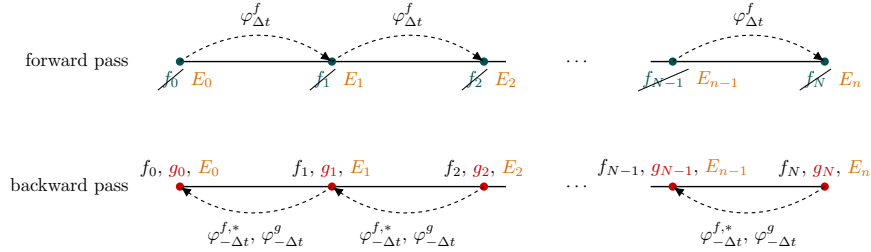

However, there are two features of the DLRA that need to be addressed in order to obtain a viable numerical scheme.
\begin{enumerate}
    \item In \cite{Kieri_2016}, the adjoint of the Lie--Trotter projector-splitting and the symmetric Strang projector-splitting are derived under the assumption that $S$ is invertible. However, it turns out that in the rank-deficient case, the Lie--Trotter projector-splitting composed with its adjoint and the Strang projector-splitting are not time-reversible. Thus, the adjoint Lie--Trotter splitting no longer fulfills the defining adjoint property and Strang splitting loses its symmetric property. This is of particular relevance when we start our simulation, as the initial conditions are commonly of low rank or even rank one and thus smaller than the simulation rank, resulting in an $S$ that is not invertible. We will show in \cref{sec:rank-deficient-case} that the assumption of full rank is indeed necessary and in the rank-deficient case the projector-splitting integrator can fail to be reversible. The reason for this is that in the argument above, we do not take into account how the low-rank factors are reconstructed from K and L. In the rank-deficient case, this can introduce non-uniqueness of the low-rank factors which destroys time-reversibility. Fortunately, from a practical point of view, this only occurs at the very beginning of the simulation where it is relatively easy to mitigate.
    \item It turns out that individual trajectories (i.e.~solutions with different initial conditions) of the low-rank scheme are chaotic (i.e.~they separate quickly during the simulation). While this is not detrimental to the accuracy of the low-rank solutions obtained, in the context of optimization it has a number of important consequences that we address in \cref{sec:chaos}. We also note that this chaotic behavior precludes us from using finite differences to compute the gradient. Thus, the adjoint state method is indeed the only viable option in the DLRA context.
\end{enumerate}

\subsection{Violation of time-reversibility in the rank-deficient case\label{sec:rank-deficient-case}}

Typically, the rank of the initial condition $r_0$ is much smaller than the simulation rank $r$. A common example is the initial condition for the so-called \emph{two-stream instability} of the Vlasov--Poisson system in kinetic plasma physics, i.e.
\begin{equation}\label{eq:two-stream-ic}
    f(0, x, v) = (1 + \alpha \cos(k x)) \frac{1}{2 \sqrt{2 \pi}} \left( \exp\left( -\frac{(v - v_\mathrm{b})^2}{2} \right) + \exp\left( -\frac{(v + v_\mathrm{b})^2}{2} \right) \right),
\end{equation}
which has rank one. Although rank-deficiency is frequently encountered in numerical simulations, the theoretical results for the projector-splitting integrator assume that the coefficient matrix is non-singular and hence do not cover this case \cite{Lubich_2014,Kieri_2016}. Moreover, if we start with a rank-deficient initial condition, the numerical flow of the PDE typically increases the rank. Such an increase of the effective rank is similar to rank-adaptiveness, which is not time-reversible. This is the motivation why we want to study in this section the time-reversibility for the rank-deficient case.

Let us start with a numerical experiment. We compute the time-reversal error for the two-stream instability for a single forward step with $\varphi^1_{\Delta t}$ and a subsequent backward step with $\varphi^{1,*}_{-\Delta t}$. Since one is the adjoint of the other, both steps should cancel. This requires us to solve the Vlasov--Poisson equation (see \cref{eq:vlasov-poisson-eq}), where we use the same electric field for the forward and backward step. The simulations are performed on the domain $(x, v) \in [0, 10 \pi] \times [-6, 6]$ with periodic boundary conditions, with grid size $64 \times 64$ and for $\Delta t = 0.01$. For the time integration of the evolution equations~\eqref{eq:psi-equations}, we used the implicit midpoint scheme and employed LU decomposition for the linear systems with NumPy \cite{Harris_2020}. Thus, the time-reversal error should be close to machine precision. We use rank $r = 3$ and are thus in a rank-deficient situation since the initial condition has rank $1$.

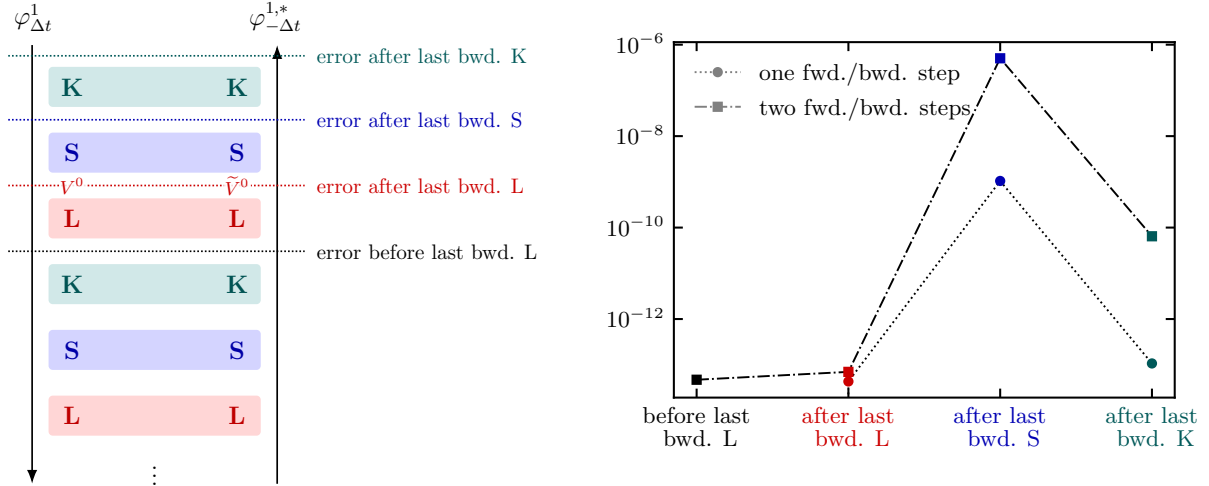
\begin{figure}[!htb]
    \centering
    \begin{minipage}[c]{0.45\textwidth}
        \centering
        \resizebox{\linewidth}{!}{\begin{tikzpicture}[
    x=1cm,
    y=1cm,
    >=Latex,
    block/.style={rounded corners=2pt, minimum width=3.55cm,
        minimum height=0.67cm, inner sep=0pt},
    steptext/.style={font=\bfseries\large},
    errorline/.style={densely dotted, line width=0.8pt}
]
    \draw[->, line width=0.8pt] (0.48,6.72) -- (0.48,-0.62);
    \draw[->, line width=0.8pt] (4.58,-0.62) -- (4.58,6.72);
    \node[font=\large] at (0.50,7.18) {$\varphi^1_{\Delta t}$};
    \node[font=\large] at (4.55,7.18) {$\varphi^{1,*}_{-\Delta t}$};

    \foreach \y/\letter/\fillcolor/\textcolor in {
        0.52/L/red!16/red!75!black,
        1.62/S/blue!17/blue!65!black,
        2.72/K/teal!18/teal!65!black,
        3.82/L/red!16/red!75!black,
        4.92/S/blue!17/blue!65!black,
        6.02/K/teal!18/teal!65!black%
    }{
        \node[block, fill=\fillcolor] (block-\letter-\y) at (2.53,\y) {};
        \node[steptext, text=\textcolor] at (1.15,\y) {\letter};
        \node[steptext, text=\textcolor] at (3.91,\y) {\letter};
    }

    \node[font=\large] at (2.53,-0.38) {$\vdots$};

    \draw[errorline, teal!70!black] (0.08,6.54) -- (5.04,6.54);
    \node[anchor=west, font=\normalsize, text=teal!70!black]
        at (5.12,6.54) {error after last bwd. K};

    \draw[errorline, blue!70!black] (0.08,5.47) -- (5.04,5.47);
    \node[anchor=west, font=\normalsize, text=blue!70!black]
        at (5.12,5.47) {error after last bwd. S};

    \draw[errorline, red!80!black] (0.08,4.37) -- (5.04,4.37);
    \node[font=\small, text=red!80!black, fill=white,
        fill opacity=0.85, text opacity=1, inner sep=0.5pt]
        at (1.15,4.37) {$V^0$};
    \node[font=\small, text=red!80!black, fill=white,
        fill opacity=0.85, text opacity=1, inner sep=0.5pt]
        at (3.91,4.37) {$\widetilde{V}^0$};
    \node[anchor=west, font=\normalsize, text=red!80!black]
        at (5.12,4.37) {error after last bwd. L};

    \draw[errorline, black] (0.08,3.27) -- (5.04,3.27);
    \node[anchor=west, font=\normalsize]
        at (5.12,3.27) {error before last bwd. L};
\end{tikzpicture}}
    \end{minipage}\hfill
    \begin{minipage}[c]{0.52\textwidth}
        \centering
        \includegraphics[width=\linewidth]{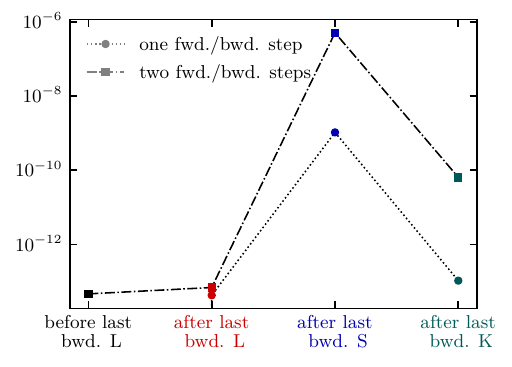}
    \end{minipage}
    \caption{Left: Schematic depiction of the errors measured between the forward integration with $\varphi^1_{\Delta t}$ and the backward integration with its adjoint, $\varphi^{1,*}_{-\Delta t}$. Right: The associated numerical time-reversal errors of the two-stream instability for one and two forward--backward steps.}
    \label{fig:r0_reverse_error_two_steps}
\end{figure}

We show the time-reversal errors between the forward and backward K, S and L substeps in \cref{fig:r0_reverse_error_two_steps}. Note that we can not compare the error between individual low-rank factors, since the basis functions are defined only up to an orthonormal matrix due to the QR decomposition. Instead, we compute the full solution $f(x, v) = \sum_{i,j=1}^r X_i(x) S_{ij} V_j(v)$ for each substep and compare the error between the forward and backward integration. This is illustrated on the left-hand side of \cref{fig:r0_reverse_error_two_steps}. 

We see from the dotted line in the right plot of \cref{fig:r0_reverse_error_two_steps} that the error after the backward S step is significantly larger than machine precision, but the backward K step seems to revert this error, with the final time-reversal error being close to machine precision. This raises the question of why such a high error in the S step occurs and why there is exact error cancellation between the backward S and K step.

Our goal is now to examine this behavior theoretically. We start by comparing the forward and backward L steps. With a time-reversible solver, such as implicit midpoint, both L steps cancel. However, due to the QR decomposition after the backward L step, the span of the $V$ basis functions might be different in the rank-deficient case, i.e.
\begin{equation}\label{eq:span-V0}
    \mathrm{span}\{V^0\} \neq \mathrm{span}\{\widetilde{V}^0\},
\end{equation}
where $\widetilde{V}^0$ are the basis functions after the backward L step, as indicated in \cref{fig:r0_reverse_error_two_steps}. Since the basis functions $X^1$ do not change during the forward and the backward S step, but the basis functions $\widetilde{V}^0$ are now different, the backward S~step with the resulting update $\widetilde{S}^*_{ij}$ is, in general, not the same as in the forward step. This explains the high time-reversal error after the backward S step in \cref{fig:r0_reverse_error_two_steps}. 

Let us now investigate why the cancellation between the backward S and K steps takes place. By multiplying \cref{eq:practical-psi-equations-S} on both sides with $X^1_i(x)$ (the updated basis functions after the forward K step) and defining $K_i^S(t, x) = \sum_{j = 1}^r X_j^1(x) S_{ji}(t)$, we obtain
\begin{equation} \label{eq:practical-psi-equations-KS} 
    \dot{K}^S_i(t, x) = -P_{\overline{X^1}} \left\langle \widetilde{V}_i^0, F\left(t, \sum_{j=1}^r K_j^S(t, x) \widetilde{V}_j^0, \theta \right) \right\rangle_v,
\end{equation}
where $P_{\overline{X^1}}$ is the projector on $\overline{X^1} = \mathrm{span}\{X_i^1\}$. During the backward step, this equation is integrated from $\Delta t$ to $0$ with the terminal condition $K^S_i(\Delta t, x) = \sum_{j = 1}^r X^1_j(x) \widetilde{S}^*_{ji}$. This is the same as \cref{eq:practical-psi-equations-K}, but with a minus sign and with an additional projector in front of the right-hand side. We obtain the following lemma.

\begin{lemma}\label[lemma]{lem:cancellation-condition}
    The K step of the projector-splitting integrator with the $\varphi^{1,*}_{\Delta t}$ splitting cancels in the rank-deficient case with the S step, if the projection in \cref{eq:practical-psi-equations-KS} is exact, i.e.~if for all $t \in [0, \Delta t]$
    \begin{equation*}
        \langle \widetilde{V}_i^0, F(t, \sum_{j=1}^r K_j^S(t, x) \widetilde{V}_j^0, \theta)\rangle_v \in \overline{X^1}, \quad i = 1, \dots, r.
    \end{equation*}
\end{lemma}

For the dotted line in \cref{fig:r0_reverse_error_two_steps}, the QR decomposition of the forward K step resulted in basis functions $X^1$ that captured the entire dynamics of the backward S step and thereby cancelled the error. However, in general, in the rank-deficient case the span of the basis generated by the QR decomposition is not unique, and we have thus no guarantee that the condition in \cref{lem:cancellation-condition} is satisfied. This is illustrated by the dash-dotted line in \cref{fig:r0_reverse_error_two_steps}, where the time-reversal error is orders of magnitude above machine precision.


We conclude that for the rank-deficient case, we can not expect time-reversibility for the projector-splitting integrator. The following theorem makes this rigorous.

\begin{theorem}\label{thm:violated-time-reversibility}
    The projector-splitting integrator in the rank-deficient case is not time-reversible.
\end{theorem}

\begin{proof}
    We show this by giving a counterexample that explicitly violates the cancellation condition of \cref{lem:cancellation-condition}. To this end, we consider the linear transport problem
    \begin{equation}\label{eq:transport-pde}
        \partial_t f(t, x, v) = \cos v \, \partial_x f(t, x, v),
    \end{equation}
    with $(x, v) \in \Omega_x \times \Omega_v = [0, 2\pi] \times [-\pi, \pi]$ and with periodic boundary conditions. In the following, we will only regard the case for rank $r = 2$ and use the rank-$1$ initial condition
    \begin{equation*}
        f_0(x, v) = f(0, x, v) = \frac{1}{\pi} \sin x \cos v.
    \end{equation*}
    
    We again use the projector-splitting integrator with Lie splitting and perform one forward step with $\varphi^1_{\Delta t}$ and one backward step with the adjoint $\varphi^{1,*}_{-\Delta t}$. We then compare the initial condition $f_0$ with $\widetilde{f}_0 = \varphi^{1,*}_{-\Delta t} \circ \varphi^1_{\Delta t}(f_0)$. For time-reversibility, they have to be equal.

    The rank-deficiency leaves us freedom in choosing the second, ``non-active'' basis functions. We set
    \begin{equation}\label{eq:initial-condition-counterexample}
        X^0 = \left( \frac{\sin x}{\sqrt{\pi}}, \, \sqrt{\frac{5}{2}} \frac{(x - \pi)^2}{\pi^{5/2}} \right), \qquad V^0 = \left( \frac{\cos v}{\sqrt{\pi}}, \, \frac{\sin v}{\sqrt{\pi}} \right), \qquad S^0 = \begin{pmatrix} 1 & 0 \\ 0 & 0 \end{pmatrix}.
    \end{equation}

    We start the forward step by computing the right-hand side for the K step according to \cref{eq:psi-equations-K}. This yields
    \begin{equation*}
        \left\langle V^0, F\left(t, \sum_{j = 1}^r K_j V_j^0, \theta \right) \right\rangle_v = \frac{1}{\pi} \int_{-\pi}^\pi \begin{pmatrix} \cos v \\ \sin v \end{pmatrix} \cos v \left( \cos v \, \partial_x K_1 + \sin v \, \partial_x K_2 \right) \mathrm{d}v = \begin{pmatrix} 0 \\ 0 \end{pmatrix}.
    \end{equation*}
    Thus, $V^0$ is orthogonal to the right-hand side of \cref{eq:transport-pde} and consequently, both the K and S step do not update $X^0$ and $S^0$. This also occurs in practical simulations. One way to mitigate this effect there is to generate basis functions directly from the right-hand side of the PDE, so that the basis already contains directions in which the solution starts to evolve. For our counterexample, this situation, however, is convenient (and in fact we did choose the initial low-rank factors to accomplish this), as it keeps the forward step as simple as possible.

    The final L substep of $\varphi^1_{\Delta t}$ cancels with the first L step of the adjoint flow $\varphi^{1,*}_{-\Delta t}$. As we know already from \cref{eq:span-V0}, the resulting $\widetilde{V}^0$ of the backward L step can differ from $V^0$. Here, we assume that the QR decomposition generates
    \begin{equation*}
        \widetilde{V}^0 = \left( \frac{\cos v}{\sqrt{\pi}}, \frac{1}{\sqrt{2 \pi}} \right).
    \end{equation*}
    Note that $\widetilde{V}^0$ is no longer orthogonal to the right-hand side of \cref{eq:transport-pde}. Consequently, the evolution equation for the backward S is now different from the forward one. There is however still the possibility that the backward S and K step cancel, provided the condition of \cref{lem:cancellation-condition} is fulfilled. 
    
    Let us therefore consider the backward S step in more detail. Since $X$ and $S$ did not change during the forward integration and during the backward L step, we have for the terminal condition of \cref{eq:practical-psi-equations-KS}
    \begin{equation*}
        K_i^S(\Delta t, x) = \sum_{j=1}^r X_j^0(x) S_{ji}^0 = \left(\frac{\sin x}{\sqrt{\pi}}, \, 0 \right).
    \end{equation*}
    Therefore, the unprojected right-hand side of \cref{eq:practical-psi-equations-KS} at time $\Delta t$ is
    \begin{equation*}
        \left\langle \widetilde{V}_i^0, F\left( \sum_{j} K_j^S(\Delta t, x) \widetilde{V}_j^0 \right) \right\rangle_v = \left( 0, \, \frac{\cos x}{\sqrt{2 \pi}} \right).
    \end{equation*}
    Clearly, this is not in the span of $X^1 = X^0$ and thus, we have a violation of the cancellation condition of \cref{lem:cancellation-condition}. 

    Let us continue with the backward S step, for which we obtain the following matrix ODE after computing the inner products of the right-hand side of \cref{eq:practical-psi-equations-S}:
    \begin{equation*}
        \frac{\mathrm{d}}{\mathrm{d}t} S(t) = -\frac{2 \sqrt{5}}{\pi^2} \begin{pmatrix} -S_{22}(t) & -S_{21}(t) \\ \phantom{-} S_{12}(t) & \phantom{-} S_{11}(t)
        \end{pmatrix}.
    \end{equation*}
    The terminal condition of the backward S step is $S(\Delta t) = S^0$, as $S$ was not changed so far by the forward integration and the backward L step. Thus, the matrix ODE has the solution
    \begin{equation*}
        S(t) = 
        \begin{pmatrix}
            \cos \left( \frac{2 \sqrt{5}}{\pi^2} (\Delta t - t) \right) & 0 \\
            0 & \sin \left( \frac{2 \sqrt{5}}{\pi^2} (\Delta t - t) \right)
        \end{pmatrix}.
    \end{equation*}
    As we integrate the backward S step from $\Delta t$ to $0$, the final result is $\widetilde{S}^* = S(0)$.
    Since our considerations were so far independent of the time step size, for simplicity we can choose $\Delta t = \pi^3 / \sqrt{5}$. This leaves the $S$ unchanged, i.e.~$\widetilde{S}^* = S^0$ and yields the simple terminal condition $K_i(\Delta t, x) = \sum_{j = 1} X_j^0(x) S_{ji}^0$ for the final backward K step. After carrying out the inner product over $v$ at the right-hand side of \cref{eq:practical-psi-equations-K}, the two components of the evolution equation for K decouple into two wave equations
    \begin{equation*}
        \partial_{tt} K_i(t, x) = \frac{1}{2} \partial_{xx} K_i(t, x),\quad i = 1, 2.
    \end{equation*}
    Integrating from $\Delta t$ to $0$ and using the periodic boundary conditions and the terminal condition from above, we obtain the solution
    \begin{equation*}
        K(0, x) = \frac{1}{\sqrt{\pi}} \left( \sin x \cos \left( \frac{\Delta t}{\sqrt{2}} \right),\, -\cos x \sin \left( \frac{\Delta t}{\sqrt{2}}  \right)  \right).
    \end{equation*} 
    The full solution $\widetilde{f}_0(x, v) = \sum_{i=1}^r K_i(0, x) \widetilde{V}^0_i(v)$ is then
    \begin{equation*}
        \widetilde{f}_0(x, v) = \frac{1}{\pi} \left( \sin x \cos v \cos \left( \frac{\Delta t}{\sqrt{2}} \right) - \frac{1}{\sqrt{2}} \cos x \sin \left(\frac{\Delta t}{\sqrt{2}} \right) \right).
    \end{equation*}
    Since $\sin(\Delta t/\sqrt{2}) \neq 0$ for the chosen $\Delta t$, we have $f_0 \neq \widetilde{f}_0$, as desired.
\end{proof}

In practice, we expect that a large time-reversal error due to rank-deficiency is only present for the first time steps. During the simulation, the rank-deficient, inactive basis functions are bit by bit activated by the right-hand side of the PDE, and $S$ eventually becomes non-singular. Thus, the projector-splitting integrator becomes time-reversible after this warm-up phase. This can be clearly seen in \cref{fig:r0_reverse_error}, where we plot the time-reversal error after full steps for the two-stream instability in the linear phase (with $t \in [0, 10]$). We used the same parameters as for \cref{fig:r0_reverse_error_two_steps}, but for performance reasons, we computed the linear systems of the implicit midpoint rule with the SciPy implementation of GMRES \cite{Virtanen_2020,Saad_1986}. 

Since the time-reversal error due to rank-deficiency is mainly present during the first time steps, it is relatively easy to mitigate. We simply store the initial condition (which in most cases is available analytically) and use this value and recompute the first time steps when computing the gradient.

\begin{figure}[!htb]
    \centering
    \includegraphics{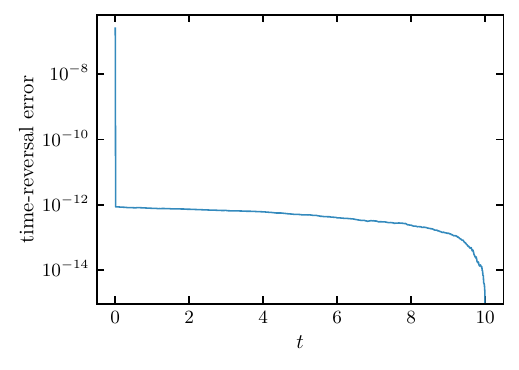}
    \caption{Time-reversal error for the two-stream instability in the time interval $[0, 10]$.}
    \label{fig:r0_reverse_error}
\end{figure}

\subsection{Chaotic behavior\label{sec:chaos}}

In this section we study the sensitive dependence on the initial conditions of the projected PDE, i.e.~\cref{eq:projected-PDE}, that is not intrinsic in \cref{eq:generic_pde}. If two solutions $f$ and $\widetilde{f}$ are initially separated by the small distance $\delta_0 = \widetilde{f}_0 - f_0$, then at a later time, the separation is given by
\begin{equation*}
    \lVert \delta(t) \rVert \sim \exp \left( \lambda_\mathrm{max}(f_0) t \right) \lVert \delta_0 \rVert,
\end{equation*}
where $\lambda_\mathrm{max}(f_0)$ is the \emph{maximal Lyapunov exponent} \cite{Strogatz_2015} and $\lVert \cdot \rVert$ is the $L^2$-norm of the phase space $\Omega$. A positive maximal Lyapunov exponent for non-periodic solution trajectories indicates that the system is chaotic.

We compute the Lyapunov exponent numerically with the method of \cite{Benettin_1976}. To this end, we start with two initial conditions $f_0$ and $\widetilde{f}_0$ which are close, i.e., $\delta_0 = \widetilde{f}_0 - f_0$ and $\lVert \delta_0 \rVert \ll 1$. Next, we integrate both initial conditions from $0$ to $\Delta t$ to obtain the updates $f_{1}$ and $\widehat{f}_{1}$. After this time step, the initial perturbation has changed, $\delta_1 = \widehat{f}_1 - f_1$. Since we are interested in the amplification of the perturbation $\lVert \delta_0 \rVert$ for each time point, we rescale the perturbed solution by $\widetilde{f}_1 = f_1 + \delta_1 \lVert \delta_0 \rVert / \lVert \delta_1 \rVert$. This requires a rank truncation, if the solution is computed with DLRA. If we continue this procedure, we obtain the sequence of perturbations $\{\delta_i\}$, $i = 1, 2, \dots$, from which we can compute an estimate of the maximal Lyapunov exponent,
\begin{equation}\label{eq:lyapunov-estimate}
    k_n(\Delta t, f_0, \delta_0) = \frac{1}{n \Delta t} \sum_{i = 1}^n \log \left( \frac{\lVert \delta_i \rVert}{\lVert \delta_0 \rVert} \right).
\end{equation}
It was shown in \cite{Benettin_1976} that for a randomly chosen $\delta_0$, we can identify $k(\Delta t, f_0, \delta_0) = \lim_{n \to \infty} k_n(\Delta t, f_0, \delta_0)$ with the Lyapunov exponent $\lambda_\mathrm{max}(f_0)$ with an error that tends to zero with $\lVert \delta_0 \rVert$. 

In \cref{fig:r1_lyapunov_exponent}, we show the results for the estimate of Lyapunov exponent computed with the projector-splitting integrator for the two-stream instability. For all ranks, except for the full rank case where $r=64$, the projector-splitting integrator has a large, positive Lyapunov exponent. For the full rank case, the projector-splitting captures the exact, non-chaotic solution, due to the exactness property of \cite[Theorem~4.1]{Lubich_2014}. We, in particular, note that there is still chaotic behavior even for $r=63$.

\begin{figure}[!htb]
    \centering
    \includegraphics{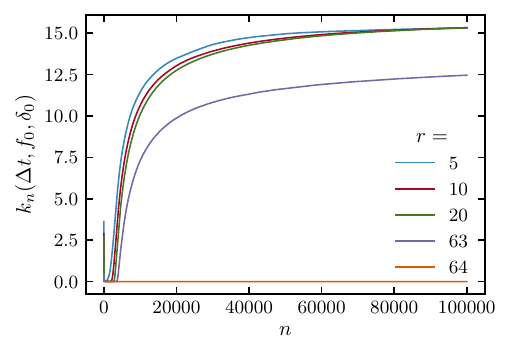}
    \caption{Estimates $k_n(\Delta t, f_0, \delta_0)$ of the maximal Lyapunov exponent for the two-stream instability example computed with the projector-splitting integrator with time step size $\Delta t = 0.01$ and initial perturbation size $\lVert \delta_0 \rVert = 10^{-8}$. The nonlinear phase of the two-stream instability starts at approximately $n = 4000$, where we also observe a higher time-reversal error due to the large (time-)local Lyapunov exponent (cf. \cref{fig:r3_time_reversal_error}).}
    \label{fig:r1_lyapunov_exponent}
\end{figure}


To our knowledge, this behavior has not been observed before. It is interesting to note that the chaotic behavior does not diminish the accuracy of the low-rank scheme. While small perturbations of trajectories separate exponentially, they stay within the low-rank approximation error of the true solution. However, it has important consequences for the optimization problem we consider here, as we will outline below.

Before proceeding, let us also note that this is not purely a shortcoming of the projector-splitting integrator. In fact, we have also conducted numerical experiments with the augmented BUG integrator \cite{Ceruti_2022a} and observed similar results. Thus, it seems to equally affect the commonly used robust DLRA integrators. However, it is not an intrinsic property of the low-rank approximation, as using the best-approximation of rank $r$, i.e.~truncating the full-rank numerical solution at each time step using an SVD, shows no chaotic behavior. Because of this, we conjecture that step-and-truncate low-rank methods (see, e.g., \cite{Kormann_2015,Guo_2024,Einkemmer_2025b}) might be immune from this behavior. However, for the present application, such methods have, in general, higher memory requirements, and it is also less clear how to make them time-reversible (as truncation is a fundamental part of these algorithms).

We also note that performing the simulation with the mass- and momentum-conservative scheme of \cite{Einkemmer_2023} did not improve the sensitivity. Therefore, we conclude that the violation of conservative properties of the underlying PDE by the DLRA is not a viable explanation for the large Lyapunov exponents.

\subsubsection{Growth of the time-reversal error}

The most immediate consequence of the large Lyapunov exponent is that even if our dynamics to reconstruct $f$ is perfectly time-reversible, there is an exponential increase in the time-reversal error due to the perturbation introduced by finite precision arithmetics. This can be clearly observed in \cref{fig:r3_time_reversal_error}, where we store a checkpoint of the forward solution every $\eta$ time steps and then perform the backward integration starting from these checkpoints. It is also interesting to note that the error growth is not uniform in time. The time-reversal error grows much faster during the nonlinear phase of the two-stream instability starting at $t = 40$. This is consistent with \cref{fig:r1_lyapunov_exponent} if we consider $k_n$ as an approximation of the (time-)local Lyapunov exponent.

From a practical point of view, the checkpointing procedure outlined in \cref{alg:memory-efficient-adjoint-method} gives a straightforward solution to this problem at the cost of some additional storage. Thus, we take a sufficient number of checkpoints during the forward pass, such that the time-reversal error stays sufficiently small. As we will see in \cref{sec:algorithm-numerical-experiments}, only a very small number of checkpoints are required in order to obtain accurate gradients for our optimization algorithm.

\subsubsection{Rough optimization landscapes}

In this section, we address the effect of DLRA on the optimization landscape. Since the low-rank approximation removes small singular values, one could expect that it has a smoothing effect on the landscape. However, a consequence of the chaotic nature of the DLRA dynamics is that the loss function is also sensitive with respect to changes in the parameters $\theta$. Indeed, we observe in our numerical experiments, that the low-rank approximation of the landscapes can be considerably rougher than those computed with a full-rank scheme. This can be seen in \cref{fig:r2_landscape_comparison}, where we show the loss landscapes computed by the projector-splitting integrator with different ranks $r$ for the beam shaping example described in \cref{sec:beam-shaping} with two parameters, $K = \{1, 2\}$.

\begin{figure}[!htb]
    \centering
    \includegraphics{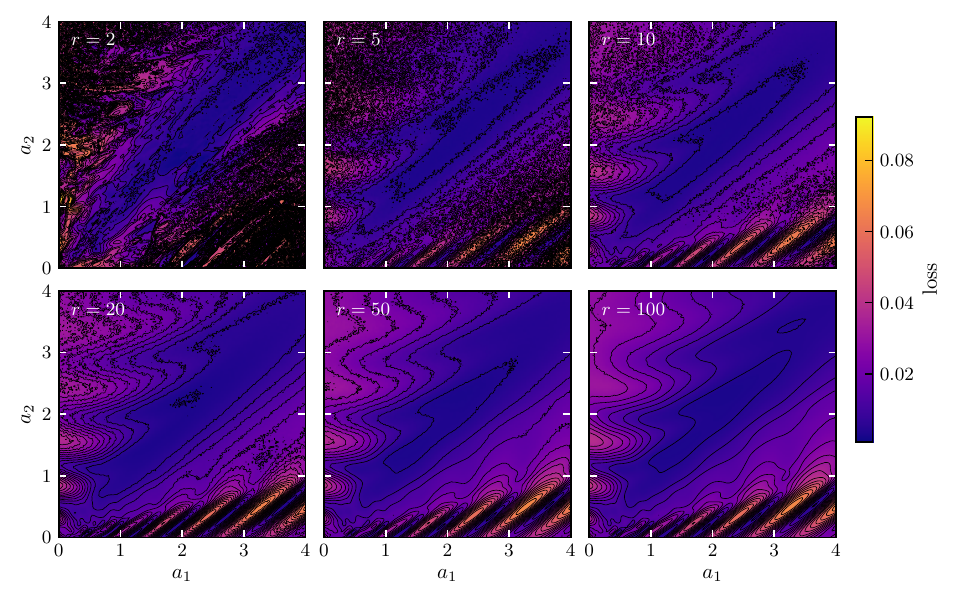}
    \caption{Comparison of the loss landscapes for the beam shaping example from \cref{sec:beam-shaping} with two parameters, $K = \{1, 2\}$. The landscapes were computed with the projector-splitting integrator for different ranks $r$. The full rank for this example is $r = 128$.}
    \label{fig:r2_landscape_comparison}
\end{figure}

Similar to what has been observed in the previous section, small perturbations of the parameters can lead to small differences in the final solutions even if all of those solutions are within the low-rank approximation error. This results in high-frequency perturbations of the loss, making the corresponding landscape rough even though the global structure is captured very well. This roughness makes local optimization difficult, as it introduces many spurious local minima and reduces the reliability of the gradient information. We will discuss this more in \cref{sec:algorithm-numerical-experiments}, where we will use optimization algorithms (such as the Adaptive Moment Estimation (Adam) optimizer \cite{Kingma_2015}) that can deal with such loss functions.

\subsubsection{Failure of finite difference approximations to the gradient\label{sec:fd-approximations}}

A further consequence of the chaotic behavior of the DLRA is that computing the gradient using finite differences gives extremely inaccurate results. We consider the parameter set \textbf{A} of the beam shaping example from \cref{sec:beam-shaping} and show in the left subplot of \cref{fig:r4_comparison_adjoint_finite_difference_gradient} the angle between the full-rank gradient (i.e., the gradient computed by \cref{eq:finite-differences} with the full-rank forward solution $f$ and with $\epsilon = 10^{-7}$) and the gradient computed with the same finite difference scheme, but where $f$ was computed with DLRA for different ranks. We observe that even for high ranks the gradient essentially points in an arbitrary direction. In particular, if the angle exceeds $90^\circ$ (as indicated by the dashed line), the gradient descent step moves along the wrong direction and the optimization no longer converges to the minimum. The right subplot shows the norm of the gradient computed with finite differences, which is up to five orders of magnitude larger than the full-rank gradient.

This is in stark contrast to the DLRA adjoint state method, where the gradient is well aligned with the full-rank solution and whose norm is similar to the one of the full-rank solution, see \cref{fig:r4_comparison_adjoint_finite_difference_gradient} (``adjoint''). Thus, for DLRA using the adjoint state method is, at least for the problems considered here, the only reliable choice to obtain accurate gradients for optimization. Note that to guarantee a fair comparison, we store for this example in our adjoint method the entire forward trajectory to compute the gradient. Thus, both methods, the finite difference and the adjoint state method, do not have a time-reversal error for this comparison.

\begin{figure}[!htb]
    \centering
    \includegraphics{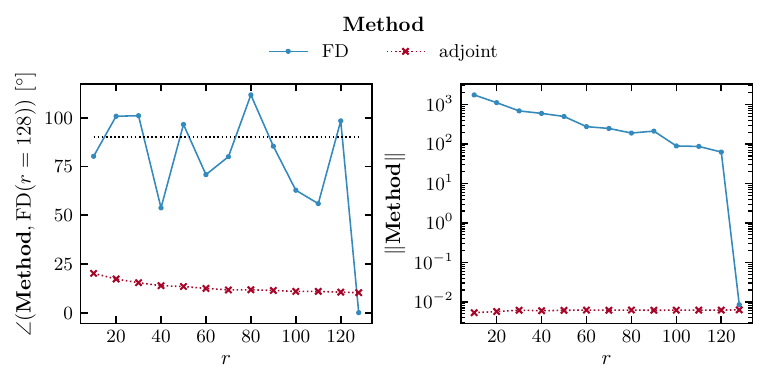}
    \caption{Left: Angle between the reference gradient and the gradient computed with the memory-efficient adjoint method (``adjoint'') and the finite difference method (``FD'', cf. \cref{eq:finite-differences}, where $f$ is computed with DLRA and $\epsilon = 10^{-7}$) depending on the rank $r$ for the initial parameter set \textbf{A} of the beam shaping example from \cref{sec:beam-shaping}. To guarantee a fair comparison, we store for this example in our adjoint method the entire forward trajectory to compute the gradient. The reference gradient was computed with the finite difference method for the full-rank forward solution $f$. Right: Norm of the adjoint and finite difference gradients for varying ranks $r$.}
    \label{fig:r4_comparison_adjoint_finite_difference_gradient}
\end{figure}

\section{Algorithm and numerical experiments}\label{sec:algorithm-numerical-experiments}

As we have seen, computing gradients by finite differences gives inaccurate results for DLRA. However, we also face for the adjoint method two main challenges: namely, the increase in time-reversal error due to the large Lyapunov exponent and the rough solution landscapes. In this section, we address the former by checkpointing and the latter by the appropriate choice of the optimization algorithm. We will show that the resulting algorithm can be used to optimize problems in kinetic plasma physics, where the forward problem is the Vlasov--Poisson equation with either external fields or source terms. In this case, the adjoint equation is similar to the forward problem and can be treated in a straightforward manner. We will consider a beam heating problem, where the modulation of the beam (a source term) is the parameter, and a beam shaping problem, where the external electric field is the parameter.

\subsection{Algorithm}
In order to control the time-reversal error of the backward pass, we modify our memory-efficient adjoint method by storing checkpoints during the forward pass to reset the time-reversal error in the backward pass. The full memory-efficient adjoint method with checkpointing is shown in \cref{alg:memory-efficient-adjoint-method}.

The main parameter to control the time-reversal error is the compression ratio $\eta$, which determines the number of time steps between two forward solutions that are stored as checkpoints in line~\ref{alg:line:store}. During the backward pass, the checkpoints are loaded (line~\ref{alg:line:load}) and the gradient and the loss functional are computed via \cref{eq:generic_gradient,eq:abstract-loss}. Thus, \cref{alg:memory-efficient-adjoint-method} can be incorporated into any gradient-based optimization method. Instead of the Lie--Trotter splitting $\varphi^1$ and its adjoint $\varphi^{1,*}$, the symmetric Strang splitting $\varphi^2$ can be used for both the forward and the backward pass.

\begin{algorithm}
    \caption{Memory-efficient adjoint state method with checkpointing}\label{alg:memory-efficient-adjoint-method}
    \begin{algorithmic}[1]
        \Require Parameters $\theta$, number of time steps $N$, final time $T$, initial condition $f_0$ and 
        \Statex terminal condition $g_N$ in the low-rank format~\eqref{eq:dlra}, compression ratio $\eta$.
        \Ensure Gradient $\nabla_\theta J$, loss $J(f_\theta)$.
        \State Store checkpoint $f_0$.
        \State Set $\Delta t = T / N$.
        \For{$i = 1, \dots, N$} \Comment{\emph{Forward pass}}
            \State Set $t_i = i \Delta t$.
            \State Compute $f_i = \varphi^1_{\Delta t}(f_{i-1})$ (\cref{alg:projector-splitting-lie}) by solving the forward problem~\eqref{eq:generic_pde}.
            \State Store quantities that are needed to reconstruct nonlinear terms.
            \If{$i \bmod \eta = 0$}
                \State Store checkpoint $f_i$. \label{alg:line:store}
            \EndIf
        \EndFor
        \State Set $\nabla_\theta J = 0$.
        \State Set $J = \Phi(f_N)$.
        \For{$i = N - 1, \dots, 0$} \Comment{\emph{Backward pass}}
            \State Set $t_i = i \Delta t$.
            \If{$i \bmod \eta = 0$}
                \State Load checkpoint $f_i$. \label{alg:line:load}
            \Else
                \State Reconstruct nonlinear terms from quantities stored in the forward pass.
                \State Compute $f_i = \varphi^{1,*}_{-\Delta t}(f_{i+1})$ by solving the forward problem~\eqref{eq:generic_pde}.
            \EndIf
            \State Compute $g_i = \varphi^{1}_{-\Delta t}(g_{i+1})$ by solving the adjoint problem~\eqref{eq:generic_adjoint}.
            \State Set $\nabla_\theta J \gets \nabla_\theta J + \langle g_i, \nabla_\theta F(t_i, f_i, \theta) \rangle \Delta t$.
            \State Set $J \gets J + \mathcal{J}(f_i) \Delta t$.
        \EndFor
    \end{algorithmic}
\end{algorithm}

The choice of the compression ratio $\eta$ is important. It should be small enough so that the time-reversal error is below a prescribed threshold, but large enough to store as few checkpoints as possible in order to save memory. In fact, $\eta$ is the checkpoint interval and corresponds to the factor of memory saved if we compare the classical adjoint method with the memory-efficient variant. We show the time-reversal error of the Strang projector-splitting integrator for different values of $\eta$ and two time integrators (classical 4th-order Runge--Kutta and the implicit midpoint rule) for the two-stream instability in \cref{fig:r3_time_reversal_error}. We observe that storing every $2000$th step (which amounts to $5$ checkpoints for the entire simulation) is sufficient to end up with a maximal time-reversal error of about $10^{-7}$. 

It is interesting to note that the chaotic behavior clearly dominates the difference between the two time integrators. Even though the implicit midpoint rule, which is time-reversible, generally has a smaller time-reversal error compared to RK4 (which is not time-reversible), this is a relatively small effect. This is particularly pronounced during the nonlinear phase of the two-stream instability starting at time $t \approx 40$, where the slope of the time-reversal error in \cref{fig:r3_time_reversal_error} becomes significantly steeper.

Using the time-reversal error to estimate the required compression ratio $\eta$ is not feasible in practical applications, since it requires storing the entire forward trajectory and comparing it against the forward problem integrated backward in time. However, since we observe in our numerical experiments that the chaotic behavior of the projector-splitting integrator is the dominant contribution of the time-reversal error for longer simulations, the approximation of the Lyapunov exponent $k_n(\Delta t, f_0, \delta_0)$, \cref{eq:lyapunov-estimate}, can be used as an estimator to determine when to store a checkpoint.

\begin{figure}[!htb]
    \centering
    \includegraphics{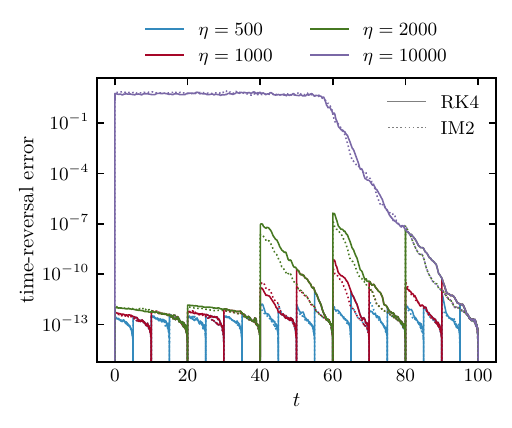}
    \caption{Time-reversal error of the Strang projector-splitting integrator for the two-stream instability in the time interval $[0, 100]$. We use different values of the compression ratio $\eta$ and for the time integration of the K, S and L evolution equations, the classical 4th-order Runge--Kutta method (RK4) and the time-reversible implicit midpoint rule (IM2). The nonlinear phase starts at approximately $t = 40$, where we also observe a higher time-reversal error due to the large (time-)local Lyapunov exponent (cf. \cref{fig:r1_lyapunov_exponent}).}
    \label{fig:r3_time_reversal_error}
\end{figure}

\subsection{Model}
We test the proposed method on optimization problems for the Vlasov--Poisson equation from kinetic plasma physics. Kinetic effects are essential in plasma physics, as many instability mechanisms can not be captured accurately by fluid models. This is also the reason why kinetic models are increasingly used for parameter identification or optimal control, which can be posed as a PDE-constrained optimization problem. In recent years, PDE-constrained optimization has been carried out for a number of kinetic equations, such as the radiative transport equation \cite{Egger_2015,Baumann_2025}, the Boltzmann equation \cite{Albi_2014,Caflisch_2021}, the Fokker–Planck equation \cite{Annunziato_2013,Fleig_2017}, the Vlasov--Poisson equation \cite{Einkemmer_2025a,Glass_2003,Glass_2012,Bartsch_2024,Knopf_2018a,Knopf_2018b,Knopf_2020}, the Vlasov--Maxwell equation \cite{Crouseilles_2025}, and drift-kinetic equations \cite{Guerra_2026}.

The main downside of the kinetic approach is the curse of dimensionality: The governing equations are posed in an up to six-dimensional phase space (as both the spatial and the velocity part can have up to three dimensions) and the memory requirements as well as the computational effort of solving these equations scales exponentially with the number of dimensions.

Therefore, low-rank approximations have received considerable attention for kinetic equations in plasma physics, beginning with \cite{Kormann_2015,Ehrlacher_2017,Einkemmer_2018}. We refer for an overview of low-rank methods for kinetic equations to the recent review article \cite{Einkemmer_2025b}. Due to the large memory footprint and the relevance in the design of future fusion devices, PDE-constrained optimization for the Vlasov--Poisson system is an important benchmark for our memory-efficient adjoint method.

In the following, we will study the optimization problems posed in \cite{Einkemmer_2024a,Einkemmer_2024b}.

\subsubsection{Vlasov--Poisson equation}
We consider a plasma with a constant, positively charged ion background, whose electrons are described by a distribution function $f(t, x, v)$, where $x$ and $v$ denotes position and velocity, respectively.  The time evolution is governed by the one-dimensional non-dimensionalized Vlasov--Poisson equation,
\begin{equation}
    \partial_t f + v\partial_x f - (E(f) + H_\theta) \partial_v f = \mathcal{S}_\theta.
    \label{eq:vlasov-poisson-eq}
\end{equation}
Here, $E(f)$ is the self-consistent electric field generated by the plasma and whose computation we will describe below, $H_\theta$ is a parameter-dependent external electric field, and $\mathcal{S}_\theta$ is a parameter-dependent source term modeling an injected beam.

For simplicity and to keep computational cost low, we will focus here on the $1+1$-dimensional case. However, the algorithm can be applied without change to higher-dimensional problems also. We solve \cref{eq:vlasov-poisson-eq} on the time interval $[0,T]$ and on the phase-space domain
\begin{equation*}
    \Omega_x \times \Omega_v = [0,L] \times \left[-V, V\right],
\end{equation*}
with periodic boundary conditions:
\begin{align*}
    f(t, x + L, v) &= f(t, x, v),\\
    f(t, x, v + V) &= f(t, x, v - V).
\end{align*}

Since the additional source term models a beam that injects additional particles, we define the charge density as
\begin{equation*}
    \rho_f(t, x) = \frac{1}{L} \int_{\Omega_x} \int_{\Omega_v} f(t, x, v) \, \mathrm{d}v \, \mathrm{d}x - \int_{\Omega_v} f(t, x, v) \, \mathrm{d}v.
\end{equation*}
This definition ensures charge neutrality, i.e., $\int_{\Omega_x} \rho_f(t, x) \, \mathrm{d}x = 0$. The electric field is then computed self-consistently from the electric potential $\phi$ which is the solution of the Poisson equation
\begin{align*}
    E(f) &= - \partial_x \phi,\\
    -\partial_{xx} \phi &= \rho_f.
\end{align*}
Another way to write this is with the Green's function $G$ that satisfies
\begin{equation}\label{eq:Greens-function}
    \partial_{xx} G(x) = \delta(x) - \frac{1}{L}.
\end{equation}
Here, we exploited the fact that our system is translation-invariant, and thus, the Green's function depends only on one variable. Moreover, our definition of the Green's function differs from the conventional one by the additional $-1/L$-term. This term ensures consistency with the boundary conditions, since integrating the left-hand side of \cref{eq:Greens-function} yields $\partial_x G(L) - \partial_x G(0)$, which has to vanish due to the periodic boundary conditions. In particular, this requires that the zeroth Fourier mode of the Green's function vanishes.

The electric potential is the convolution of $G$ with the charge density, and hence, we can conveniently express the electric field by
\begin{equation}
    E(f)(t, x) = G' * \rho_f(t, \cdot),
    \label{eq:E-Green}
\end{equation}
where we use the shorthand $G' = \partial_x G$. Note that we never use the analytical expression for $G$ in real space, but solve the electric field and the related convolution terms efficiently in Fourier space.

\subsubsection{Adjoint equation}
We now derive the adjoint equation for the Vlasov--Poisson system. To this end, we start by rewriting \cref{eq:vlasov-poisson-eq} in the form $\partial_t f = F(t, f, \theta)$ of \cref{sec:theoretical-background}. We have
\begin{equation}
    F(t, f, \theta) = -v \partial_x f + (E(f) + H_\theta) \partial_v f + \mathcal{S}_\theta.
    \label{eq:vp-operator}
\end{equation}

For the adjoint equation~\eqref{eq:generic_adjoint}, we first have to compute the Fréchet derivative of the right-hand side operator. Since the source term does not depend on $f$, we obtain with the chain rule
\begin{equation}\label{eq:VP-Frechet}
    D_f(F)(h) = -v \partial_x h + (E(f) + H_\theta) \partial_v h + D_f(E)(h) \partial_v f,
\end{equation}
where $h$ is again an arbitrary variation. The electric field $E$ depends linearly on $f$ via \cref{eq:E-Green}, therefore, the last term of \cref{eq:VP-Frechet} is
\begin{equation*}
    D_f(E)(h) \partial_v f = (\partial_x G * \rho_h) \partial_v f,
\end{equation*}
where $\rho_h$ is the charge density with $f$ replaced by the variation $h$.
Next, we compute the adjoint of the Fréchet derivative, $(D_f(F))^*$, by using its definition
\begin{equation*}
    \left\langle g, D_f(F)(h) \right\rangle_{x, v} = \left\langle (D_f(F))^*(g), h \right\rangle_{x, v}.
\end{equation*}
For the first two terms of \cref{eq:VP-Frechet}, we move the operators from $h$ onto $g$ via integration by parts. The last term requires more effort: Since the Green's function of the Poisson equation is symmetric, we exploit
\begin{equation*}
    \partial_x G(x - y) = -\partial_y G(y - x),
\end{equation*}
and write the last term of \cref{eq:VP-Frechet} as
\begin{equation}
    \left\langle g, (G' * \rho_h) \partial_v f \right\rangle_{x, v} = -\left\langle G' * \langle g, \partial_v f \rangle_v, \rho_h \right\rangle_x.\label{eq:last-term-VP-Frechet}
\end{equation}
With the definition of the charge density $\rho_h$, we obtain for the adjoint of the Fréchet derivative
\begin{equation}
    (D_f(F))^*(g) = v \partial_x g - (E(f) + H_\theta) \partial_v g - \int_{\Omega_x} G' * \left\langle g, \partial_v f \right\rangle_v \left( \frac{1}{L} - \delta(x - x') \right) \, \mathrm{d}x'.
\end{equation}

If we define the adjoint source term $\mathcal{S}_g$ as
\begin{equation*}
    \mathcal{S}_g(t, x) = -\nabla_f \mathcal{J}(t, x) + \int_{\Omega_x} G' * \langle g,\partial_v f \rangle_v (t, x') \left( \frac{1}{L} - \delta(x - x') \right) \, \mathrm{d}x',
\end{equation*}
then the adjoint equation for the Vlasov--Poisson equation takes the compact form
\begin{equation*}
    \partial_t g + v \partial_x g - \left( E(f) + H_\theta \right) \partial_v g = \mathcal{S}_g.
\end{equation*}
This is very similar to the Vlasov--Poisson equation with a source term, cf. \cref{eq:vlasov-poisson-eq}. In fact, we use in our implementation the same code for solving the forward and the adjoint problem.

\subsubsection{Time-reversibility of the low-rank steps\label{sec:low-rank-steps-reversibility}}
In this section, we describe how time-reversibility for the Vlasov--Poisson equation can be efficiently achieved. The standard approach to solve the Vlasov--Poisson equation is to freeze the electric field at the beginning of each time step, so $E(t, x) \approx E_0(x) = G' * \rho_f(0, \cdot)$. After $f_0(x, v)$ is advanced according to the Vlasov--Poisson equation, the electric field is updated at the end of the time step with \cref{eq:E-Green}.

This can be analogously employed for the first-order Lie--Trotter projector-splitting integrator by freezing the electric field at the beginning of a time step and updating it after the last subflow. A second-order method for the Strang projector-splitting integrator can be achieved with the predictor--corrector scheme of \cite{Einkemmer_2014a,Einkemmer_2014b,Einkemmer_2018} as follows. We first perform a half step with Lie--Trotter splitting $\varphi^1_{\Delta t / 2}$ using the frozen electric field $E_0(x)$ and then obtain $f_{1/2}(x, v)$ at $\Delta t / 2$. We use this update of the distribution function to compute $E_{1/2}(x)$ from \cref{eq:E-Green}. Finally, we perform a full step with Strang splitting $\varphi^2_{\Delta t}$ using the frozen $E_{1/2}(x)$.

If we apply this scheme for both the Lie--Trotter splitting and its adjoint, the resulting scheme is however not time-reversible. Using the above outlined first-order scheme, the first L step of the backward integration with the adjoint Lie--Trotter splitting is solved with the frozen electric field at the terminal value $f_N$, whereas for the forward integration with Lie--Trotter splitting, the final L step is integrated with the frozen electric field of $f_{N-1}$. The same argument holds for Strang splitting, since the electric field at the half step $E_{1/2}(x)$ computed with the predictor--corrector method differs for the forward and the backward integration. A proper time-reversible scheme would have to solve the K, S and L evolution equations and the Poisson equation implicitly, which is however very expensive.

Therefore, we follow the standard approach of freezing the electric field, but we reuse the electric field from the forward integration. For the K step for example, we write \cref{eq:practical-psi-equations-K} as
\begin{equation*}
    \partial_t K = \mathrm{RHS}_K(K, E^*),
\end{equation*}
where we use $E^* = E_{i-1}$ for Lie--Trotter splitting and $E^* = E_{i - 1/2}$ for Strang splitting during the $i$-th time step ($i = 1, \dots, N$). For the backward integration, we reuse the electric fields stored from the forward integration and thus ensure that the right-hand side is exactly the same for the backward integration. The electric field is lower-dimensional, as it only depends on $x$, and is therefore cheap to store.

Thus, the K, S and L substeps can be solved individually with a symmetric scheme such as implicit midpoint, and we circumvent the need of solving all evolution equations and the Poisson equation together with an implicit scheme. Alternatively, we could also employ an explicit time integration scheme. Higher-order explicit methods (such as classical 4th-order Runge--Kutta) show only a small time-reversal error while at the same time, they are very cheap to compute.

\subsubsection{Implementation}
We implemented our algorithm in Python and used JAX for the computational bottlenecks. JAX is a library for array-oriented numerical computation on both CPUs and GPUs, and for just-in-time compilation \cite{Deepmind_2020}. Moreover, it allows automatic differentiation, which we used for computing the gradients $\nabla_\theta I_\theta(t)$ and $\nabla_\theta H_\theta(t, x)$ for the problems in \cref{sec:beam-heating,sec:beam-shaping}, respectively.

For the optimization algorithm, we used the JAX libraries Optax and Equinox \cite{Deepmind_2020,Kidger_2021}. These libraries yield efficient implementations of optimizers such as gradient descent with backtracking line search or the Adam. All computations, except the loss landscapes, were carried out on an Apple~M4 CPU. The loss landscapes were computed on a GTX~1080Ti GPU.

Since the exact time-reversibility can not be guaranteed due to the rank-deficiency of our rank-$1$ initial conditions and the chaotic behavior of the projector-splitting integrator, we used our memory-efficient checkpointing scheme with the non-reversible classical Runge--Kutta scheme of order $4$ (RK4) and $30$ substeps for the time integration of the K, S and L evolution equations. This is justified by \cref{fig:r3_time_reversal_error}, where the difference between RK4 and the time-reversible implicit midpoint rule is small. We favor RK4, because it is significantly more computationally efficient. Throughout the numerical experiments, we use the projector-splitting integrator with Strang splitting.

\subsection{Beam heating}\label{sec:beam-heating}
In this numerical example, we study the heating of a plasma through an injected beam. The beam can drive a microinstability causing an additional peak in the velocity distribution. Since the peak is superimposed on the tail of the Maxwellian equilibrium, the instability is called the \emph{bump-on-tail instability}. The bump-on-tail instability is not only relevant for beam heating \cite{Speth_1989}, but also for ion-cyclotron heating \cite{Yamagiwa_1987} and Alfv\'en-wave destabilization that can degrade reactor performance \cite{Van_Zeeland_2021}. 

It was experimentally shown in \cite{Van_Zeeland_2021} that modulation of the beam can significantly affect the ion transport even when the total injected heating power is unchanged. Therefore, \cite{Albi_2025} investigated whether a beam modulated with feedback control could be used to suppress instabilities in the plasma. Such an approach is however challenging to realize because the relevant time scales are of the order of the inverse plasma frequencies or Alfv\'en times and thus very short.

Therefore, several control strategies without a feedback loop have been investigated in the literature. In \cite{Nakamura_1970}, it was experimentally demonstrated that modulating an electron beam close to the natural oscillation frequency of the system reduces the electric field strength during the instability. Similar results were found in \cite{Fukumasa_1982}, through the nonlinear coupling of a relatively small number of modes. A related problem was investigated in \cite{Qin_2014}, where the velocity of the beam was modulated. The authors developed a linear theory showing that a sinusoidal variation of the velocity of the beam reduces the growth rate of the instability. In \cite{Einkemmer_2024a,Einkemmer_2025a}, external electric fields have been optimized to suppress the two-stream and bump-on-tail instabilities. In \cite{Chen_2025}, instabilities are controlled by an external field that effectively eliminates unstable roots in the dispersion relation of the uniformly magnetized Vlasov-Poisson system. Finally, \cite{Einkemmer_2024b} optimized the intensity modulation for a one-dimensional Vlasov--Poisson beam-heating model. In the following, we use the setup of this paper to test our memory-efficient adjoint method.

To this end, we use the time-integrated total electric energy as the loss functional, i.e.
\begin{equation*}
    J(f) = \frac{1}{2} \int_0^T \int_{\Omega_x} E(f)(t, x)^2 \, \mathrm{d}x \, \mathrm{d}t.
\end{equation*}
Comparing this loss functional with the generic form~\eqref{eq:abstract-loss}, we only have a running contribution, $\mathcal{J}(f) = J(f)$, and the terminal contribution $\Phi(f(T))$ for this example is zero. Since the electric field grows exponentially in case of the bump-on-tail instability until nonlinear effects become strong enough to lead to saturation, the time-integrated electric energy is a good measure for the severity of the plasma instability. Numerical experiments in \cite{Guerra_2025} indicate that this functional works very well for optimization of the Vlasov--Poisson equation.

For the source term of the adjoint equation, we need the Riesz representative of the Fréchet derivative $D_f(\mathcal{J})(h)$. From a calculation that is very similar to the one carried out in \cref{eq:last-term-VP-Frechet}, we obtain
\begin{equation*}
    \nabla_f \mathcal{J}(t, x) = -\int_{\Omega_x} G' * E(f)(t, \cdot) \left( \frac{1}{L} - \delta(x - x') \right) \, \mathrm{d}x'.
\end{equation*}
Since the loss functional of this example has no terminal contribution $\Phi$, the terminal condition for the adjoint equation~\eqref{eq:generic_terminal_condition} becomes
\begin{equation*}
    g(T, x, v) = 0.
\end{equation*}

Following \cite{Einkemmer_2024b}, we consider here the case with vanishing external electric field, $H_\theta(x) = 0$. The particle beam is modelled by a source term that is Gaussian in space (centered at the middle of the domain $x_0$ and with spatial extent $\sigma$) and Maxwellian in velocity (with average beam velocity $\bar{v}_{\mathrm{b}}$ and thermal velocity $v_{\mathrm{th,b}}$ of the beam) and is switched off after time $t_\mathrm{b}$, i.e.
\begin{equation}\label{eq:source}
    \mathcal{S}_\theta(t, x, v) = I_\theta(t) M(x, v) \vartheta(t_{\mathrm{b}} - t),
\end{equation}
with 
\begin{equation*}
    M(x, v) = \frac{1}{\sqrt{2 \pi \sigma^{2}}} \exp{\left( -\frac{(x - x_0)^{2}}{2 \sigma^{2}} \right)} \frac{1}{\sqrt{2 \pi v_{\mathrm{th,b}}^{2}}} \exp{\left( -\frac{(v - \bar{v}_{\mathrm{b}})^{2}}{2 v_{\mathrm{th,b}}^{2}} \right)},
\end{equation*} 
and $I_\theta(t)$ is the parameter-dependent modulated beam amplitude and $\vartheta(t_{\mathrm{b}} - t)$ the Heaviside step function. Note that due to the non-vanishing source term $\mathcal{S}_\theta$, we have for the distribution function that $\int_{\Omega_x} \int_{\Omega_v} f(t, x, v) \, \mathrm{d}v \, \mathrm{d}x = 1 + \int_0^t I_\theta(s) \, \mathrm{d}s$.

For this numerical experiment, we assume a sinusoidal beam profile $I_\theta(t) = I_0 + b_1 \sin(\omega t)$, where $I_0 = 0.05$, and $b_1$ and $\omega$ are the model parameters, $\theta = (b_1, \omega)$. The gradient of the loss functional with respect to the parameters can then be computed from \cref{eq:generic_gradient} as
\begin{equation*}
    \nabla_\theta J = \int_0^T \nabla_\theta I_\theta(t) \langle M(x,v), g(t, x, v) \rangle_{x,v} \, \mathrm{d}t.
\end{equation*}
For the space discretization of the K and L steps, we used second-order centered differences on a grid with $1024$ spatial and $512$ velocity points, respectively. The time integration is carried out until $T = 80$, with time step size $\Delta t = 0.01$. We assume that the distribution function is initially constant in $x$ and Maxwellian in $v$,
\begin{equation*}
    f_0 = f(0, x, v) = \frac{1}{\sqrt{2 \pi}} \exp\left( -\frac{v^{2}}{2} \right).
\end{equation*}
All model parameters are listed in \cref{tab:beam-heating-parameters}. In particular, $T$ is much larger than $t_{\mathrm{b}}$ to ensure that the plasma remains stable after switching the beam off.

\begin{table}[!htb]
    \begin{center}
        \begin{tabular}{cccc} \toprule
            \multicolumn{2}{c}{\thead{model parameters}} & \thead{initial parameters $\theta_{\mathrm{init}}$} & \thead{optimized parameters $\theta_{\mathrm{opt}}$} \\ \midrule
            $\begin{aligned}
                L &= 4 \pi\\
                V &= 6 \\
                I_0 &= 0.05 \\
                \sigma &= 0.2
            \end{aligned}$ &
            $\begin{aligned}
                t_{\mathrm{b}} &= 25 \\
                v_{\mathrm{th,b}} &= 0.5 \\
                x_0 &= L/2 \\
                v_{\mathrm{b}} &= 3.5 
            \end{aligned}$ &
            $\begin{aligned}
                b_1 &= -0.02 \\
                \omega &= 1.35 \\
                & \\
                &
            \end{aligned}$ &
            $\begin{aligned}
                b_1 &= -0.033 \\
                \omega &= 1.491 \\
                & \\
                &
            \end{aligned}$
            \\ \bottomrule
        \end{tabular}
    \end{center}
    \caption{The model parameters, the initial parameters $\theta_{\mathrm{init}}$ and the optimized parameters $\theta_{\mathrm{opt}}$ for the beam heating example.}\label{tab:beam-heating-parameters}
\end{table}

The optimization is done with gradient descent using backtracking line search with the Armijo condition. We set the learning rate to $0.25$, the Armijo control parameter to $10^{-4}$ and the backtracking decrease factor to $0.5$. Convergence is achieved when the difference between new and old parameters is below $10^{-3}$. We start the optimization with initial parameters $\theta_{\mathrm{init}}$ tabulated in \cref{tab:beam-heating-parameters}. The gradient is computed with the memory-efficient adjoint method using checkpoints, with rank $r = 20$ and for different values of the compression ratio $\eta$. Since the optimization with gradient descent is not scale-invariant and the scales for frequency and beam amplitude differ vastly, we worked with normalized parameters during the optimization to ensure faster convergence \cite{Nocedal_2006}.

In the left subplot of \cref{fig:e1_0_landscape_trajectory}, we show the trajectory of the optimized parameters in the two-dimensional loss landscape for $\eta = 100$ (which corresponds to $80$ checkpoints). We observe that our method accurately computes the gradient and finds the minimum (highlighted by a star-shaped marker) within $11$ generations. In the right subplot of \cref{fig:e1_0_landscape_trajectory}, we compare the electric energy of a uniform beam with $b_1 = 0$ and the beam with the optimized parameters $\theta_\mathrm{opt}$. The modulated beam effectively suppresses the electric energy, and it remains small even after the beam is switched off at $t_\mathrm{b} = 25$ (highlighted by the dashed vertical line). We also observe that the optimized beam parameters and the behavior of the electric energy are in very good agreement with the findings of \cite{Einkemmer_2024b}.

\begin{figure}[!htb]
    \centering
    \includegraphics{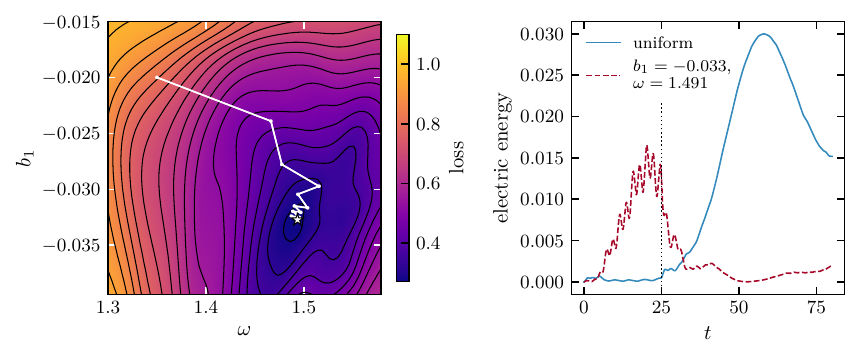}
    \caption{Left: Parameter trajectory for the memory-efficient adjoint method with the compression ratio $\eta = 100$ for the beam heating example. We used gradient descent with backtracking line search and the Armijo condition for the optimization. The converged result is highlighted with a star-shaped marker. Right: Electric energy for a uniform beam with $b_1 = 0$ and for a beam with the optimized parameters $\theta_\mathrm{opt}$. The time $t_\mathrm{b} = 25$ when the beam is switched off is highlighted by a dashed vertical line.}
    \label{fig:e1_0_landscape_trajectory}
\end{figure}

In \cref{fig:e1_1_loss_reversal_error}, we compare the loss and the maximal time-reversal error over generations for different values of the compression ratio $\eta$. Remarkably, our algorithm is very robust with respect to the time-reversal error. Even saving only every $4000$th step (i.e.~only a single checkpoint) in the forward pass still achieves convergence. 

As expected, the time-reversal error increases with the compression ratio $\eta$, but the number of generations needed for convergence is not monotonously increasing with $\eta$. If we consider for example $\eta=2000$, convergence is attained within $7$ generations. For large time-reversal errors, our method seems to resemble a stochastic gradient descent that sometimes can achieve faster convergence than the classical gradient descent. This is similar to the ``partially-converged approach'' described in \cite{Giles_2000}, where partially converged forward and adjoint solutions are used to compute the gradient, thus keeping the cost per design cycle relatively low.

The algorithm breaks down for $\eta = 8000$, where only the initial condition is stored as a checkpoint. For this maximally possible value of $\eta$, the algorithm gets instantly trapped in a wrong minimum.

\begin{figure}[!htb]
    \centering
    \includegraphics{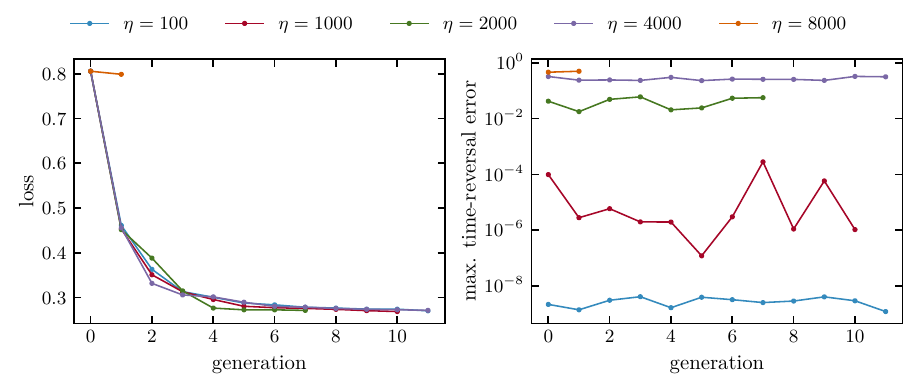}
    \caption{Loss (left) and maximal time-reversal error (right) over generations for the memory-efficient adjoint method with rank $r = 20$ for different values of the compression ratio $\eta$ for the beam heating example. We used gradient descent with backtracking line search and the Armijo condition for the optimization.}
    \label{fig:e1_1_loss_reversal_error}
\end{figure}

\cref{tab:memory-requirements} shows that our approach successfully reduces the memory requirements by approximately a factor of $500$ for all values of the compression ratio $\eta$. The reason why the memory requirements for increasing $\eta$ are not reduced further lies in the self-consistent electric field that is stored at every time point of the forward pass. Already for $\eta = 100$, the electric field contributes to almost one third of the memory requirements. However, we expect even more drastic memory reductions for higher-dimensional phase spaces, where the lower-dimensional electric field is a significantly smaller fraction of the overall memory cost.

\begin{table}[!htb]
    \begin{subtable}[t]{0.48\textwidth}
        \centering
        \caption*{\textbf{Beam heating}}
        \begin{tabular}{rcc} \toprule
            \thead{$\eta$} & \thead{memory req. \\ (in MB)} & \thead{compression} \\ \midrule
            $100$       & $85.95$     & $2.56 \cdot 10^{-3}$    \\
            $1000$      & $68.03$     & $2.03 \cdot 10^{-3}$    \\
            $2000$      & $67.03$     & $2.00 \cdot 10^{-3}$    \\
            $4000$      & $66.53$     & $1.98 \cdot 10^{-3}$    \\
            $8000$      & $66.28$     & $1.97 \cdot 10^{-3}$    \\ \bottomrule
        \end{tabular}
    \end{subtable}
    \hfill
    \begin{subtable}[t]{0.48\textwidth}
        \centering
        \caption*{\textbf{Beam shaping}}
        \begin{tabular}{rcc} \toprule
            \thead{$\eta$} & \thead{memory req. \\ (in MB)} & \thead{compression} \\ \midrule
            $1$     & $3.70$     & $0.34$    \\
            $2$     & $1.94$     & $0.18$    \\
            $5$     & $0.88$     & $0.08$    \\
            $10$    & $0.52$     & $0.05$    \\ \bottomrule
        \end{tabular}
    \end{subtable}
    \caption{Memory requirements for the beam heating (left) and the beam shaping (right) examples using the memory-efficient adjoint method with rank $r = 20$. We assume for the forward pass in total $T / (\eta \Delta t)$ checkpoints and that the electric field at every time point has to be stored. The compression is the memory ratio between our memory-efficient method and the classical adjoint method, which requires the full forward trajectory to be stored.}\label{tab:memory-requirements}
\end{table}

\subsection{Beam shaping}\label{sec:beam-shaping}
For this numerical example, we consider the focusing problem studied in \cite{Einkemmer_2024a}, where the goal is to find parameters of an external electric field that maintains a specific localized shape of the distribution function (this is related to, e.g., beam shaping problems). A related problem from the semiconductor industry is the identification of doping profiles from a voltage-to-current map in the Vlasov--Poisson system \cite{Cheng_2011}.

To this end, we consider the loss functional
\begin{equation}\label{eq:loss-functional-2}
    J(f) = \frac{1}{2} \lVert f(T) - f_0 \rVert^2,
\end{equation}
where we take the $L^2$ norm in both the spatial and velocity spaces to measure the difference of the terminal value of $f$ against the initial condition $f_0$,
\begin{equation*}
    f_0 = f(0, x, v) = \frac{1}{\sqrt{2 \pi}} \exp\left( -a (x - b)^{2} \right) \sin^{2}\left( \frac{x}{2} \right) \exp\left( -\frac{v^{2}}{2} \right),
\end{equation*}
where $a = 0.2$ and $b = 2 \pi$. This initial condition is plotted in \cref{fig:e2_2_initial_shape} and corresponds to two spatially concentrated beams with the velocity distributed according to the Maxwellian. 
\begin{figure}[!htb]
    \centering
    \includegraphics{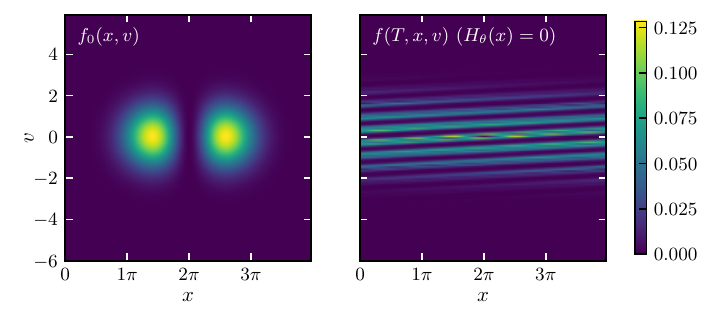}
    \caption{Left: Initial condition $f_0(x, v)$ of the beam shaping example. Right: Distribution function $f(T, x, v)$ for $T = 20$, without external electric field $H_{\theta}(x)$.}
    \label{fig:e2_2_initial_shape}
\end{figure}

Without an external electric field, the distribution function exhibits after $T = 20$ a very filamented structure, as can be seen from the right subplot of \cref{fig:e2_2_initial_shape}. Particles with a higher velocity are looped back to the domain due to the periodic boundary conditions, which causes the filamentation. We want to find an external electric field that keeps the distribution function as close as possible to the initial value. In this example, we maintain the shape solely with an external electric field and set the source term to zero, $\mathcal{S}_{\theta} = 0$. For the external electric field, we assume a decomposition in sinusoidal modes as also used in \cite{Einkemmer_2024a}, i.e.
\begin{equation*}
    H_{\theta}(x) = -\sum_{k = 1}^{10} \theta_k \sin \left(\frac{k x}{2} \right).
\end{equation*}
Thus, we optimize now in total ten parameters $\theta_1,\dots,\theta_{10}$. Such a relatively high-dimensional parameter space is the typical use case for adjoint methods, where computing the gradient with finite differences becomes very expensive. For the beam heating example with only two parameters, computing the gradient with finite differences using DLRA would have been also feasible from the perspective of computational costs. However, we have seen already in \cref{sec:fd-approximations} that the adjoint method does not only have a better scaling with respect to the number of parameters, but its gradient also shows better convergence than the gradient computed with the DLRA finite difference method.

With the loss function~\eqref{eq:loss-functional-2}, we have the complementary setting of the previous example: Now, we do not have a running contribution to the loss functional, but only a terminal contribution, so $\Phi(f(T)) = J(f)$ and $\mathcal{J}(f) = 0$. Therefore, the contribution of the loss functional to the source term is zero, but we have a non-zero terminal condition for the adjoint by \cref{eq:generic_terminal_condition}. We compute the terminal condition from the Fréchet derivative of the terminal contribution, $D_f(\Phi)(h(T))$, and obtain
\begin{equation*}
    g(T, x, v) = \nabla_{f(T)} \Phi = f(T, x, v) - f_0(x, v).
\end{equation*}
Since $f_0$ has rank $1$ and $f(T)$ has in general rank $r$, the difference $f(T) - f_0$ has rank $r + 1$. Thus, in order to obtain the low-rank representation of the terminal condition, we need to slightly truncate the rank. Finally, we compute the gradient from \cref{eq:generic_gradient} as
\begin{equation}\label{eq:gradient-ex-2}
    \nabla_\theta J = \left\langle \nabla_\theta H_\theta(\cdot), \int_0^T \langle g(t, \cdot, v), \partial_v f(t, \cdot, v) \rangle_v \mathrm{d}t \right\rangle_x.
\end{equation}

For the space discretization of the K and L evolution equations, we employ Strang splitting to separate the advection equations in space and velocity. Analogous to \cite{Einkemmer_2024a}, the resulting equations are solved with a semi-Lagrangian discontinuous Galerkin scheme using linear interpolation. The computations are performed on a grid with size $128 \times 128$ and with time step size $\Delta t = 0.25$. We choose the domain size $L = 4 \pi$, the maximal velocity $V = 6$ and the final time $T = 20$. For the following numerical experiments, we will use three different sets of initial parameters called \textbf{A}, \textbf{B} and \textbf{C} (see \cref{tab:initial_parameters}) that were obtained in \cite{Einkemmer_2024a} with a genetic optimization method.

\begin{table}[!htb]
    \begin{center}
        \begin{tabular}{lrrr} \toprule
            & \thead{\textbf{A}} & \thead{\textbf{B}} & \thead{\textbf{C}} \\ \midrule
            $\theta_1$      & $-0.6953110$     & $0.9450489$      & $-0.6953110$ \\
            $\theta_2$      & $-1.7011901$     & $1.1410373$      & $-1.7011901$ \\
            $\theta_3$      & $-3.7023607$     & $-1.5500754$     & $-3.1878159$ \\
            $\theta_4$      & $-1.0494850$     & $0.8429296$      & $0.9743365$  \\
            $\theta_5$      & $-0.4569529$     & $-0.0802972$     & $1.8268111$  \\
            $\theta_6$      & $1.8768650$      & $2.4046179$      & $1.6804664$  \\
            $\theta_7$      & $1.9196010$      & $0.9644806$      & $2.3189560$  \\
            $\theta_8$      & $1.6915317$      & $2.2524267$      & $1.6915317$  \\
            $\theta_9$      & $0.4209613$      & $-0.1217143$     & $1.3303226$  \\
            $\theta_{10}$   & $-0.4064942$     & $-0.6091703$     & $-0.8904933$ \\ \bottomrule
        \end{tabular}
    \end{center}
    \caption{The three initial parameter sets \textbf{A}, \textbf{B} and \textbf{C} of the beam shaping example, taken from \cite{Einkemmer_2024a}.}\label{tab:initial_parameters}
\end{table}

We first revisit the convergence study for the gradient from \cref{sec:fd-approximations} for the initial parameters \textbf{A}, but now we no longer store the entire forward trajectory for the adjoint method, but we use our method with the checkpoint strategy. In the left subplot, we depict the angle between the reference gradient and the gradient obtained by the adjoint method for different values of the compression ratio $\eta$. As in \cref{sec:fd-approximations}, the reference gradient was computed with full-rank finite differences. We observe that the angle for the gradient computed by the adjoint method is for all values of $\eta$ well below $90^\circ$ (note that $\eta = 40$ amounts to storing only two checkpoints for the entire solution), and the angle seems to slightly correlate with the time-reversal error shown in the right subplot of \cref{fig:e2_0_comparison_adjoint_finite_difference_gradient}. Remarkably, our adjoint method is very robust to the time-reversal error and yields gradients that are well aligned, even though the time-reversal error is large.

\begin{figure}[!htb]
    \centering
    \includegraphics{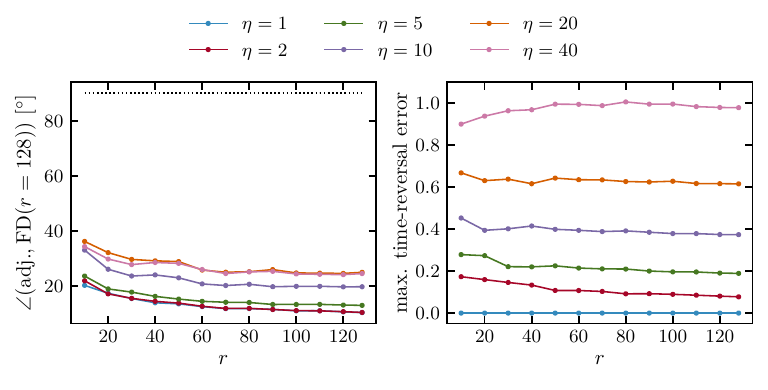}
    \caption{Left: Angle between the reference gradient and the gradient computed with the memory-efficient adjoint method with different values of $\eta$ for the initial parameter set \textbf{A} of the beam shaping example. Right: Time-reversal error of the memory-efficient adjoint method. The reference gradient was computed with the finite difference method (\cref{eq:finite-differences}) for the full-rank forward solution $f$ with $\epsilon = 10^{-7}$.}
    \label{fig:e2_0_comparison_adjoint_finite_difference_gradient}
\end{figure}

Next, we want to optimize the three different sets of initial parameters \textbf{A}, \textbf{B} and \textbf{C} with a gradient-based method. The initial parameters were obtained, as mentioned above, with a genetic algorithm, so refining the results with a gradient-based method would correspond to the second stage of the hybrid optimization method outlined in the introduction. Due to the high-dimensional and rough loss landscape, we expect to encounter many local minima. For this numerical example, we use Adaptive Moment Estimation (Adam) for optimization, since classical gradient descent is prone to getting stuck in minima. Adam is very popular in the machine learning community and modifies classical gradient descent by adaptive estimates of lower-order moments of the gradient \cite{Kingma_2015}. We compute in total $200$ parameter generations with Adam using learning rate $\alpha = 0.01$ and the gradient computed from the memory-efficient adjoint method as input.

In \cref{fig:e2_1_loss_reversal_error}, we show the loss and the maximal time-reversal error for the optimization using the gradient of our memory-efficient adjoint method with $r = 20$. We observe that for all parameter sets and all values of $\eta$, the optimization converges within $150$ generations. More specifically, the optimization improves the loss found by the global optimizer by a factor of $1.5$, $3.5$ and $4.5$ for the three parameter sets \textbf{A}, \textbf{B} and \textbf{C}, respectively. Notably, the time-reversal error for $\eta > 1$ is very high, but nevertheless, the algorithm improves the loss and, in the case for parameter set \textbf{B} and $\eta \ge 5$, even faster than for $\eta = 1$ (where the entire forward solution was stored and reused for the backward pass). Next to the subplots showing the loss, we plot the full-rank loss (i.e., the loss computed with a full-rank simulation, but using the parameters obtained by the optimization with the memory-efficient adjoint method with $r = 20$). The full-rank loss indicates that the algorithm does not only improve the loss for the low-rank approximation, but also provides a good approximation of the loss for the underlying problem.

An interesting pattern can be observed from the time-reversal error for parameters \textbf{B} and \textbf{C}: The smaller the loss, the smaller the time-reversal error. This indicates that in the vicinity of a minimum, the forward solution $f$ is low-rank, and we compute the exact solution with the projector-splitting integrator due to the exactness property. Consequently, this results in a small Lyapunov exponent and the time-reversal error is small. This makes sense since our goal is to stay as close as possible to $f_0$, which is rank $1$.

\begin{figure}[!htb]
    \centering
    \includegraphics{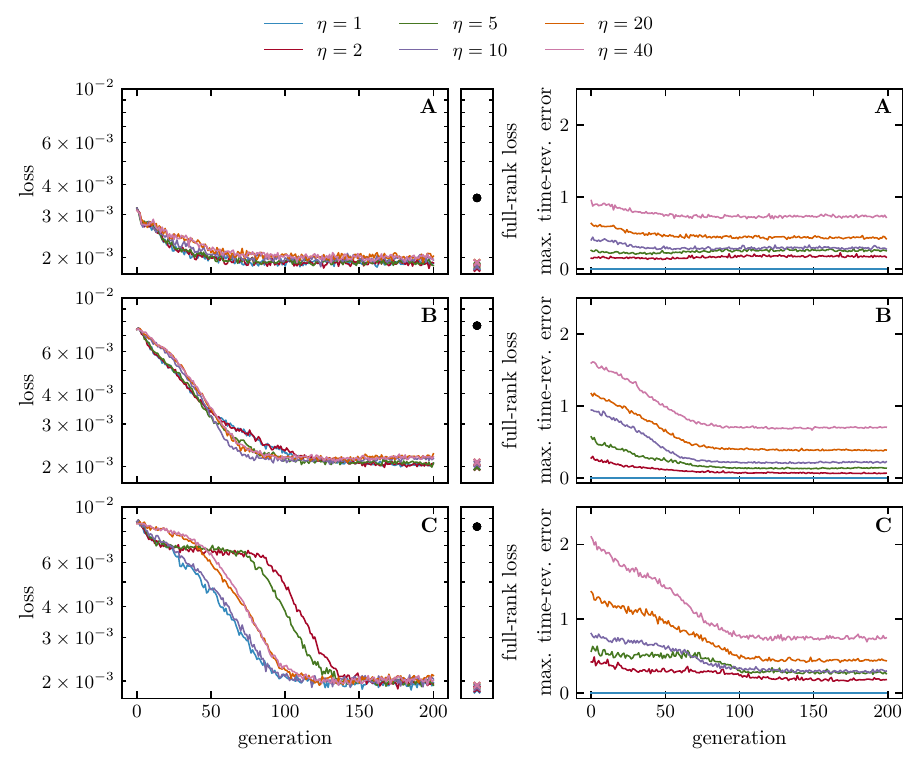}
    \caption{Loss (left) and maximal time-reversal error (right) over generations for the memory-efficient adjoint method with rank $r = 20$ and for different values of $\eta$ for the beam shaping example. In the small subpanel, the full-rank loss for the initial (black dots) and the final generation is shown. The optimization was done with Adam.}
    \label{fig:e2_1_loss_reversal_error}
\end{figure}

\begin{figure}[!htb]
    \centering
    \includegraphics{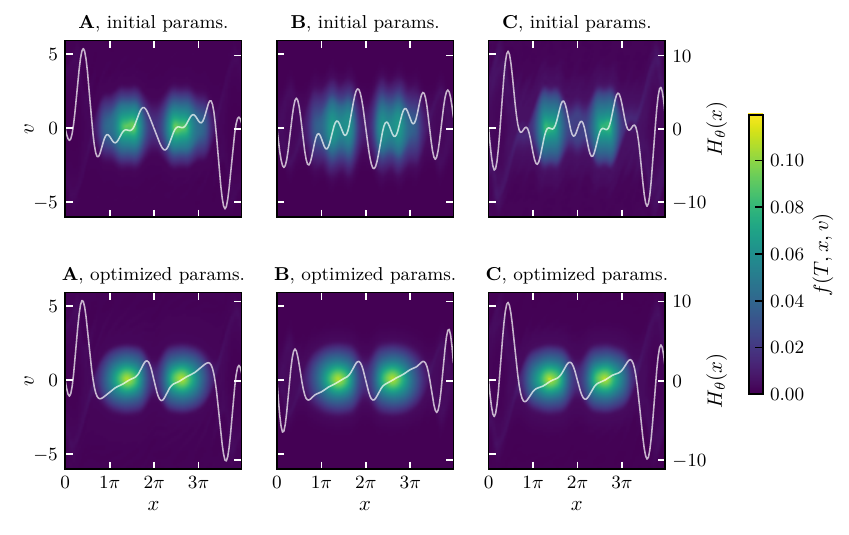}
    \caption{Distribution function $f(T, x, v)$ and external electric field $H_{\theta}(x)$ for the three initial parameter sets of \cref{tab:initial_parameters} (top row) and for the optimized parameter sets (bottom row) for the beam shaping example. The gradient, computed with the memory-efficient adjoint method with $r = 20$ and $\eta = 2$, served as input for the optimization with Adam. In order to evaluate the quality of the optimization, we used the optimized parameters and recomputed all the distribution functions for this figure with full rank.}
    \label{fig:e2_2_shape_comparison_LR_full_rank}
\end{figure}

Finally, we show in \cref{fig:e2_2_shape_comparison_LR_full_rank} the beam profiles at the final time and the external electric fields for the initial parameters (top row) and the optimized parameters obtained from the optimization with Adam and the memory-efficient adjoint method (bottom row). All beam profiles were computed with full rank. This demonstrates again that the optimization based on DLRA manages to find parameters that yield also an improved full-rank loss. We see that the optimized external electric fields maintain a better separation and focusing of the two beams. Furthermore, the optimized external electric fields are flattened and less oscillatory than the initial fields.

\section{Conclusion}\label{sec:conclusion}
In this work, we presented a memory-efficient adjoint method for the PDE-constrained optimization. This method is based on two ingredients: The first one is the dynamical low-rank approximation, which reduces the dimensionality of the solution of the PDE. The second ingredient is the time-reversibility of the projector-splitting integrator, which circumvents the need of storing the entire forward problem. We did show that the time-reversibility for the projector-splitting integrator can be violated in the rank-deficient case and observed large Lyapunov exponents and consequently, chaotic behavior of the DLRA solutions. To remedy this, we devised a checkpointing strategy that controls the time-reversal error. Finally, we showed the effectiveness of our method for two examples from kinetic plasma physics, where we have observed multiple orders of magnitude improvement in terms of memory cost.

In future research, our method could be improved in several directions. For example, one could adaptively select the times at which checkpoints are stored based on the estimate of the Lyapunov exponent given by \cref{eq:lyapunov-estimate}. This would allow us to reduce the number of checkpoints even further, especially in the linear regime. We have also seen from \cref{tab:memory-requirements} that for the memory-efficient adjoint method the quantities that are used to reconstruct the nonlinear terms of the PDE during the backward pass (i.e. the lower-dimensional electric field) can become a storage bottleneck. Although we expect this to be less relevant for high-dimensional problems, it would be interesting to investigate how to further compress these quantities (e.g.~using a low-rank tensor format that also acts in time).

It is further an interesting question whether the chaotic behavior observed for the robust DLRA integrators is an intrinsic feature of such methods or if different types of methods could be developed that overcome this limitation. If this can be done it would reduce the need for checkpointing, thus further reducing memory cost, and most likely result in smoother optimization landscapes. Finally, it would be interesting to theoretically investigate why gradient-based optimization works with our memory-efficient adjoint method even if the time-reversal error is very high, and to derive theoretical guarantees for the direction of the gradient depending on the rank.

\section{Data availability}\label{sec:data-availability}
The code for reproducing the numerical experiments and plots is available at \url{https://codeberg.org/jmangott/vp-1d-adjoint-method}.

\section{Acknowledgements}
L.E. would like to acknowledge helpful discussions with Jonas Kusch (Norwegian University of Life Sciences) on low-rank integrators and their reversibility. J.M. would like to thank Christian Bargetz (Universität Innsbruck) for helpful discussions on Fréchet derivatives. This research was funded by the Austrian Science Fund (FWF), \url{https://doi.org/10.55776/PAT2937525}.

\printbibliography[heading=bibintoc]

\end{document}